\documentclass{amsart}

\usepackage[utf8]{inputenc}
\usepackage{amsmath}
\usepackage{amssymb}
\usepackage{enumerate}
\usepackage{mathpazo}
\usepackage{amsthm}
\usepackage{tikz-cd}
\usepackage{mathrsfs}
\usepackage{amsrefs}
\usepackage{mathtools}
\usepackage[normalem]{ulem}
\usepackage{caption}

\usepackage{geometry}
\usepackage{setspace,kantlipsum} %spacing

\usepackage{thmtools} % for restating theorems without changing numbering
\usepackage{thm-restate} % for restating theorems without changing numbering
\usepackage{verbatim} %multiline commenting

\usepackage[adobe-utopia]{mathdesign} %font
\usepackage[T1]{fontenc} %font

\usepackage{tikz}
\usetikzlibrary{matrix,arrows}

\usepackage{graphicx}
\input{style.sty}

\title{Shrinking targets on square-tiled surfaces}
\author{Josh Southerland}
\address{Indiana University, Bloomington, IN 47405}
\email{jwsouthe@iu.edu}

\begin{document}

\begin{abstract}
  We study a shrinking target problem on square-tiled surfaces. We show that the action of a subgroup of the Veech group of a regular square-tiled surface exhibits Diophantine properties. This generalizes the work of Finkelshtein, who studied a similar problem on the flat torus~\cite{Fin16}.
\end{abstract}

\setstretch{1.25}

\maketitle
%\tableofcontents

 %%%%%%%%%%%%%%%%%%%%%%%%%%%%%%%
 %%%%%%%%%%%%%%%%%%%%%%%%%%%%%%%
 %%%%%%%%%%%%%%%%%%%%%%%%%%%%%%%
 %%%%%%%%%%%%%%%%%%%%%%%%%%%%%%%
 %% INTRODUCTION %%%%%%%%%%%%%%%
 %%%%%%%%%%%%%%%%%%%%%%%%%%%%%%%
 %%%%%%%%%%%%%%%%%%%%%%%%%%%%%%%
 %%%%%%%%%%%%%%%%%%%%%%%%%%%%%%%
 %%%%%%%%%%%%%%%%%%%%%%%%%%%%%%%

 \section{Introduction}\label{sec:1-Introduction}
 
 In this paper we study the action of the abundant set of derivatives of affine linear maps on a regular square-tiled surface, which is a particular type of translation surface. The set of derivatives we study is an arithmetic group, and we show that the action of subgroups of these arithmetic groups exhibit Diophantine properties.
 
 \subsection{Definition of a translation surface}
 
 A \emph{translation surface} is a pair $(X,\omega)$ where $X$ is a compact, connected Riemann surface without boundary and $\omega$ a non-zero holomorphic differential on $X$.
 
 There is an equivalent definition of a translation surface that is more intuitive: a translation surface is an equivalence class of polygons or sets of polygons in the plane $\C$ such that each edge is identified by translation to a parallel edge on the opposite side of the polygon (or opposite side of a polygon in the set of polygons). The equivalence is given by a cut-and-paste procedure that preserves the positive imaginary direction relative to the ambient $\C$. 
 
 \begin{figure}[h]
   \centering
   \def\svgwidth{\columnwidth}
   \rotatebox{0}{\scalebox{0.3}{%% Creator: Inkscape inkscape 0.92.5, www.inkscape.org
%% Accompanies image file 'STS_L_perturbed.pdf' (pdf, eps, ps)
\begingroup%
  \makeatletter%
  \providecommand\color[2][]{%
    \errmessage{(Inkscape) Color is used for the text in Inkscape, but the package 'color.sty' is not loaded}%
    \renewcommand\color[2][]{}%
  }%
  \providecommand\transparent[1]{%
    \errmessage{(Inkscape) Transparency is used (non-zero) for the text in Inkscape, but the package 'transparent.sty' is not loaded}%
    \renewcommand\transparent[1]{}%
  }%
  \providecommand\rotatebox[2]{#2}%
  \newcommand*\fsize{\dimexpr\f@size pt\relax}%
  \newcommand*\lineheight[1]{\fontsize{\fsize}{#1\fsize}\selectfont}%
  \ifx\svgwidth\undefined%
    \setlength{\unitlength}{754.42995258bp}%
    \ifx\svgscale\undefined%
      \relax%
    \else%
      \setlength{\unitlength}{\unitlength * \real{\svgscale}}%
    \fi%
  \else%
    \setlength{\unitlength}{\svgwidth}%
  \fi%
  \global\let\svgwidth\undefined%
  \global\let\svgscale\undefined%
  \makeatother%
  \begin{picture}(1,0.95481272)%
    \lineheight{1}%
    \setlength\tabcolsep{0pt}%
    \put(0,0){\includegraphics[width=\unitlength,page=1]{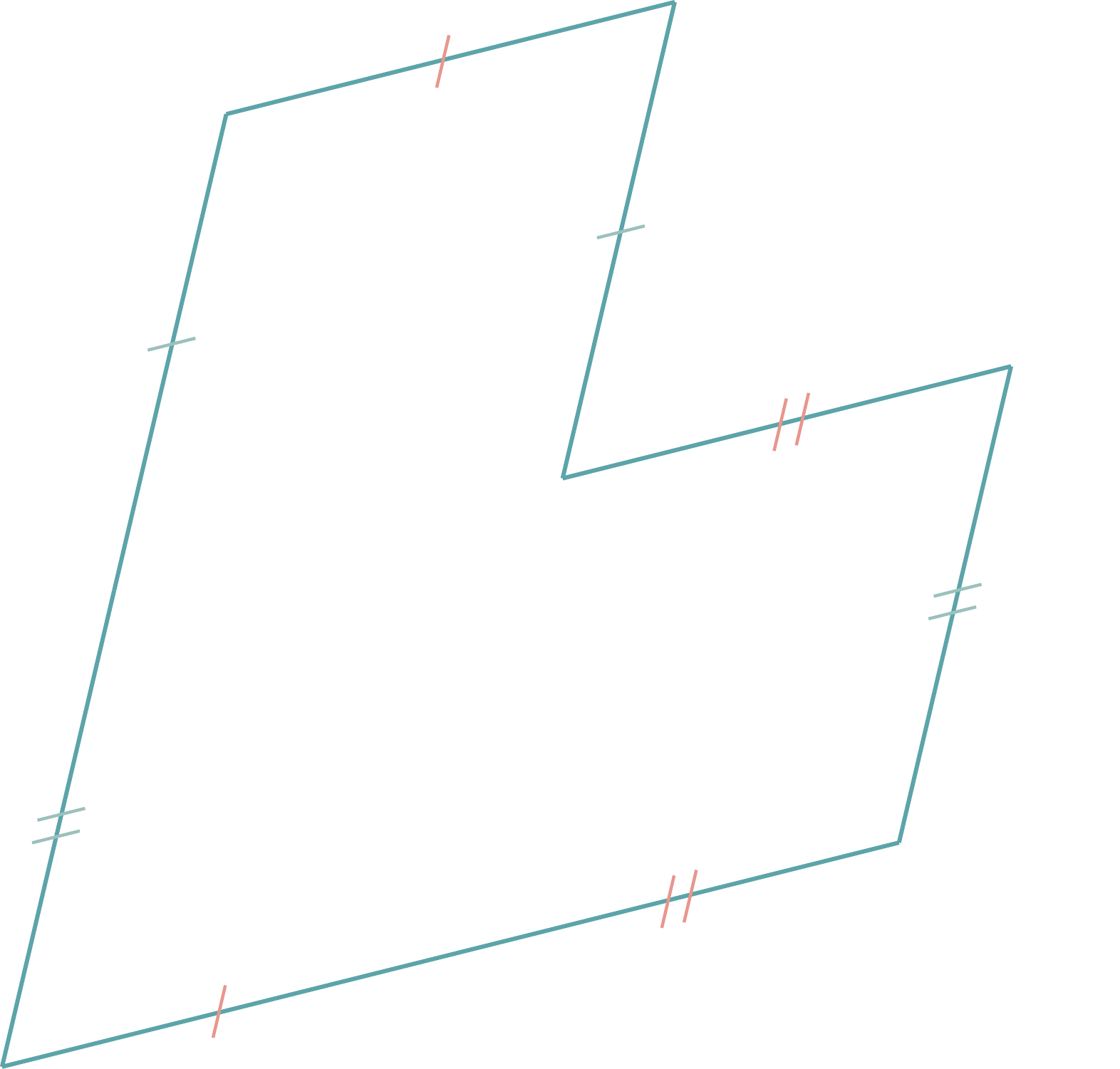}}%
    \put(0.97944757,0.52471785){\color[rgb]{0,0,0}\rotatebox{14.03624347}{\makebox(0,0)[lt]{\lineheight{1.25}\smash{\begin{tabular}[t]{l}`\end{tabular}}}}}%
  \end{picture}%
\endgroup%
}}
   \caption{Translation surface}
   \label{fig:TS} 
 \end{figure} 
 
 Note that by imposing the condition that sides are identified to opposite edges of the polygon we ensure that the positive imaginary direction is well-defined globally. Translation surfaces are flat surfaces away from a finite set of singular points, and at the singular points are cone points whose angles are integer multiples of $2\pi$. 

 \begin{definition}[Square-tiled Surface~\cite{Ma18},~\cite{Zm11}]\label{sts}
 A \emph{square-tiled surface} is a translation surface $(X,\omega)$ given by a (finite) branched cover over the square torus, $q: X \rightarrow \T^2$, branched over $0$. The one-form $\omega$ is given by the pullback of $dz$ under the covering map $q$, $\omega = q^*(dz)$. 
  \end{definition}

 Square-tiled surfaces are so named because they have a polygonal representation which looks like a tiling of squares (each square projects to the square torus in the branched cover). Consequently, square-tiled surfaces come with a natural combinatorial description: $(M,\sigma,\tau)$, where $M$ denotes the degree of the cover and $\sigma, \tau \in S_M$ are permutations that encode gluing information. $\sigma(i) = j$ means that the  right edge of the $i^{th}$ square is glued to the left edge of the $j^{th}$ square.  $\tau(i)=j$ means that the top edge of the $i^{th}$ square is glued to the bottom edge of the $j^{th}$ square.

 \begin{definition}[Square-tiled Surface~\cite{Ma18},~\cite{Zm11}]\label{sts}
 A \emph{regular} square-tiled surface is a square-tiled surface $(X,\omega)$ whose automorphism group (automorphisms of the translation structure) is transitive on the set of squares, $q^*((0,1)^2)$. 
  \end{definition}
 
 For the equivalence of the different definitions of a translation surface or square-tiled surface, the reader is encouraged to visit~\cite{Wri15},~\cite{Zm11}. %The zeroes of the holomorphic differential in the first definition correspond to the singular points, and the order of the zeroes correspond to the amount of excess angle at the singularities. %Moreover, given a compact, connected Riemann surface without a boundary and a non-zero holomorphic differential, we can find a set of charts with translations as transition functions by integrating against the form $\omega$. 

 \subsection{The $SL_2(\R)$-action}
 
 The group $SL_2(\R)$ acts on the moduli space of translation surfaces, where the action of a matrix is just the usual linear action on the polygons. Since the linear action sends parallel lines to parallel lines, the action sends a translation surface to a translation surface. In fact, the natural action in this setting is $GL_2^+(\R)$, but for our purposes, the action of $SL_2(\R)$ is more relevant since we are only interested in volume preserving maps. 

 \begin{comment}
 The geometric origins of the action can be found in the work of Thurston~\cite{Th88} and Veech~\cite{Vee89}. In Veech's influential paper, Veech answers a question posed by Gromov: do there exist translation surfaces in the moduli space which are periodic in the sense that the stabilizer of the $SL_2(\R)$-action is a lattice in $SL_2(\R)$? Veech answered in the affirmative, and such surfaces became known as lattice  (or Veech) surfaces. 
 \end{comment}
 
 For a more detailed introduction to translation surfaces, the reader is encouraged to consult one of the many excellent surveys on the topic~\cite{Che17}, ~\cite{Wri15}, ~\cite{Zor06}.
 
 \subsection{The Dynamical System} 

 \subsubsection{Action of the Veech Group}
 
 The stabilizer of the action at a translation surface $(X,\omega)$ is called the \emph{Veech group} of this surface, denoted $SL(X,\omega)$. For an example of a stabilizing element, consider the unit square with opposite sides identified by translation (a torus) and let $g = \begin{bmatrix} 1 & 1 \\ 0 & 1 \end{bmatrix}$. 

 \begin{figure}[h]
   \centering
   \def\svgwidth{\columnwidth}
   \rotatebox{0}{\scalebox{0.9}{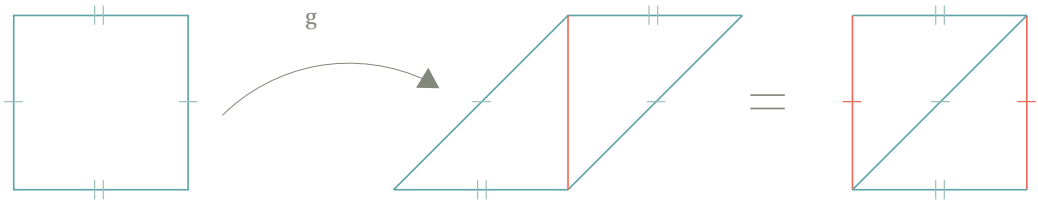}}
   \caption{Stabilizing element of the Veech group}
   \label{fig:stab} 
  \end{figure} 

 Using the cut-and-paste procedure pictured in Figure \ref{fig:stab}, we can reassemble the new polygon as the old, while respecting the ``north" direction on the surface. This example shows us that the Veech group is not always trivial. The action of the matrix appears related to linear maps on the surface, and in fact, this is true. We can identify the Veech group with the collection of \emph{derivatives} of affine linear maps on the surface~\cite{Vee89}. 
 
 There exist translation surfaces surfaces, for example, square-tiled surfaces, whose stabilizers are lattices in $SL_2(\R)$. Such lattices are necessarily discrete, non-cocompact, finite covolume subgroups of $SL_2(\R)$~\cite{HubSch04}. We call these surfaces \emph{lattice surfaces}. Veech groups of lattice surfaces contain a hyperbolic element, which can be represented as a matrix with expanding and contracting eigenspaces. The corresponding linear action of this element, after several applications, sufficiently ``mixes" the points on the surface. In fact, the map will be \emph{ergodic} (with respect to the Lebesgue measure on the surface). The existence of this element means that the action of the Veech group is \emph{ergodic}. Hence, we can ask questions about the \emph{density} of the orbits. One way to do this is to frame the question as a shrinking target problem. Fix a lattice surface $S$ with Veech group $\Gamma$, and pick any $y \in S$. Let $B_{g}(y)$ denote the open ball of area (or measure) $\phi(\lvert \lvert g \rvert \rvert)$ (a decreasing function of the operator norm). Does almost every $x \in S$ have the property that $g\cdot x \in B_g(y)$ for infinitely many $g \in \Gamma$? How fast can $\phi$ decrease (the target shrink) before this no longer holds?

 \begin{figure}[h]
  \centering
  \captionsetup{justification=centering}
  
  \def\svgwidth{\columnwidth}
%    \rotatebox{0}{\scalebox{0.5}{\input{ShrinkingTarget2.pdf_tex}}}
    \includegraphics[trim=0 100 0 0,clip,width=0.4\textwidth]{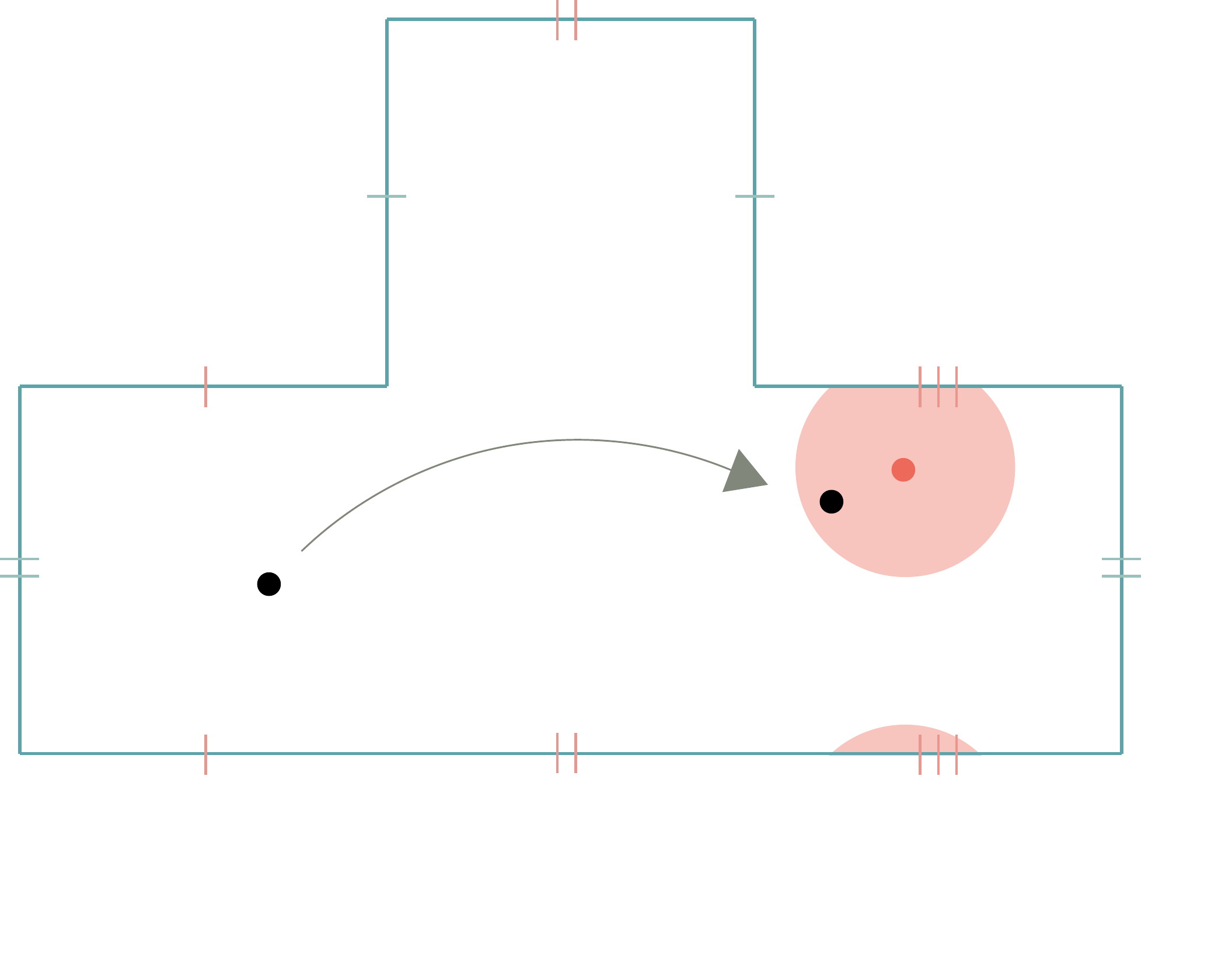}
  \caption{Hitting the target}
  \label{fig:targ} 
 \end{figure}

 In 2016, Finkelshtein studied a shrinking target problem on the square torus~\cite{Fin16}. The torus is an example of a translation surface and $SL_2(\Z)$ is its Veech group. Moreover, $SL_2(\Z)$ is a lattice subgroup of $SL_2(\R)$, so the torus is an example of a lattice surface. Finkelshtein showed that the action of $SL_2(\Z)$ on the torus exhibits certain Diophantine estimates. Finkelshtein's proof relies on a fundamental connection between the dynamics of the Veech group action and the Laplacian on the torus. 
 
 The action of the Veech group on the surface induces a group representation, the $\emph{Koopman representation}$, $\pi: SL_2(\Z) \rightarrow U(L^2(\T^2))$, where $\pi(g) f(x) = f(g^{-1}x)$. Recall that the eigenfunctions of the Laplacian, $\Delta = -(\partial_x^2 + \partial_y^2)$, are solutions to $\Delta f = \lambda f$. We can compute eigenfunctions: $e^{2\pi i m x}e^{2\pi i n y}$, where $(m,n) \in \Z^2$. Let $g = \begin{bmatrix} a & b \\ c & d \end{bmatrix} \in SL_2(\Z)$, then

 \begin{align*}
     \pi(g)e^{2\pi i m x}e^{2\pi i n y} &= e^{2\pi i (dm-cn) x}e^{2\pi i (an-bm)y} \text{.}
 \end{align*}
 
 This is significant: the Koopman representation sends eigenspaces of the Laplacian to eigenspaces. In other words, \emph{the action of the Veech group plays nicely with the spectral properties of the Laplacian.} In fact, we can say precisely how the eigenspaces are permuted by noting how $(m,n)$ is permuted: by multiplying on the left by the inverse transpose of $g$. 

% In 2016, Finkelshtein used this to decompose the Koopman representation and show that the operator norm of an averaged Koopman representation (averaged over a measure on a subgroup of $SL_2(\Z)$) is equivalent to the operator norm of an averaged left regular representation associated to the subgroup. This, coupled with spectral estimates of an induced representation on the visual boundary of $SL_2(\Z)$, enables one to prove a shrinking target property on the torus with respect to the action of \emph{any} subgroup of the Veech group $SL_2(\Z)$. 
 
 In what follows, we study the action of the Veech group on a \emph{square-tiled surface}. This problem is challenging for the following reason: the action of the Veech group on a translation surface does not, in general, respect the eigenspaces of the Laplacian. %In fact, the problem is worse. If we define the domain of the Laplacian to be the set of smooth functions with compact support  away from the singularities, the domain of the adjoint of the Laplacian will include a set of functions with poles at the singularities~\cite{Hil09}. A calculation shows that the Koopman representation sends these functions to functions outside the domain of the adjoint operator.
 
 We are able to bypass these difficulties by leveraging properties of the branched cover over the torus. Our main result shows that the action of a subgroup of a Veech group acting on a regular square-tiled surface exhibits similar Diophantine properties that are governed by the \emph{critical exponent} of the subgroup. Recall the definition of critical exponent:

 \begin{definition}[Critical Exponent, $\delta_{\Gamma}$]\label{critexp}
 Let $\Gamma$ be a Fuchsian group. The critical exponent, $\delta_{\Gamma}$, is 
 
 \begin{equation*}
 \delta_{\Gamma} := \limsup_{R \rightarrow \infty} \frac{\mathrm{log}(\#\{g \in \Gamma: d_{\Hh}(g.x_0,x_0) \leq R \})}{R} \text{,}
 \end{equation*}
 \end{definition}
 
 \noindent for any $x_0$, where $g.x_0$ denotes the action of $g$ on $x_0$ by M\"obius transformation.  $\delta_{\Gamma}$ is independent of the basepoint $x_0$.
 
 The critical exponent $\delta_{\Gamma}$ is the exponent required for convergence in the Poincar\'e series of the group $\Gamma$~\cite{Bea68},~\cite{Pat76}, which is equivalent to the exponential growth rate of the number of points in the orbit of $\Gamma$ acting on the upper half-plane~\cite{Sul79} seen in Definition \ref{critexp}.
 
 Patterson~\cite{Pat76} showed that for a finitely generated Fuchsian group $\Gamma$, the critical exponent is precisely the Hausdorff dimension of the limit set, $\Lambda = \overline{\Gamma x} \cap S^1$, where $S^1$ is the circle at infinity. Sullivan~\cite{Sul79} showed that in the general case of a Fuchsian group, the critical exponent is the Hausdorff dimension of the \emph{radial} limit set, $\Lambda_r \subset \Lambda$ consisting of all points in the limit set such that there exists a sequence $\lambda_n x \rightarrow y$ remaining within a bounded distance of a geodesic ray ending at $y$. 
 
 These various interpretations are particularly relevant to our work since we obtain Theorem \ref{thm:1} indirectly through spectral estimates of the boundary representation of the subgroup $\Gamma$. 
 
 \begin{restatable}{thm}{thmA}\label{thm:1}
        Let $(X,\omega)$ be a regular square-tiled surface, and let $\Gamma$ be a subgroup of the Veech group $SL(X,\omega)$ with critical exponent $\delta_{\Gamma} > 0$. For any $y \in X$, for Lebesgue a.e. $x \in X$, the set 
 
                \begin{equation*}
                    \{g \in \Gamma : \lvert gx - y \rvert < \lvert\lvert g \rvert\rvert^{-\alpha} \}
                \end{equation*}
 
        \noindent is 
 
                \begin{enumerate}
                    \item finite for every $\alpha > \delta_{\Gamma}$
                    \item infinite for every $\alpha < \delta_{\Gamma}$
                \end{enumerate}
        \noindent where $\lvert \lvert \, \cdot \, \rvert \rvert$ is the operator norm of $g$ (as a linear transformation on $\R^2$). 
 \end{restatable}
 
 \noindent In fact, Theorem \ref{thm:1} holds for parallelogram-tiled surfaces as well. As with Finkelshtein's result~\cite{Fin16}, our result has the added benefit that we can deduce Diophantine properties of thin subgroups of the Veech group.

 \begin{remark}
 The spectral theory of translations surfaces and, in particular, square-tiled surfaces, has been studied by Hillairet~\cite{Hil08},~\cite{Hil09}. 
 \end{remark}

 \subsection{Acknowledgements}\label{acknowledgements}
 
 The author thanks Jayadev Athreya for proposing this project and providing guidance, and to the anonymous referee for many helpful comments. Additionally, the author thanks Alexis Drouot, Dami Lee, Farbod Shokrieh, and Bobby Wilson for helpful discussions, and Chris Judge for helpful comments regarding Theorem \ref{thm:GJ}. Additionally, the author thanks Lior Silberman for identifying an error in an earlier version of this work, as well as for a series of informative discussions on the representation theory of groups of operators on singular spaces.  
 
 %%%%%%%%%%%%%%%%%%%%%%%%%%%%%%%
 %%%%%%%%%%%%%%%%%%%%%%%%%%%%%%%
 %%%%%%%%%%%%%%%%%%%%%%%%%%%%%%%
 %%%%%%%%%%%%%%%%%%%%%%%%%%%%%%%
 %% BACKGROUND & HISTORY %%%%%%%
 %%%%%%%%%%%%%%%%%%%%%%%%%%%%%%%
 %%%%%%%%%%%%%%%%%%%%%%%%%%%%%%%
 %%%%%%%%%%%%%%%%%%%%%%%%%%%%%%%
 %%%%%%%%%%%%%%%%%%%%%%%%%%%%%%%

 \section{Shrinking Targets}\label{sec:2-Shrinking}

 In this section, we will give a technical description of a shrinking target problem and identify the main obstacle that we must overcome to solve one. Throughout this section, $(X, \mathscr{B}, \mu)$ is a probability space, and $T: X \to X$ is a measure-preserving transformation, unless otherwise indicated.
 
 \subsection{Set-up}
 
 First, recall Poincar\'e recurrence.
 
 \begin{theorem}[Poincar\'e recurrence]
 Let $(X, \mathscr{B},\mu)$ be a probability space, let $T: X \rightarrow X$ be a measure preserving transformation, and let $E \in \mathscr{B}$.  Define a semigroup action on $X$ by the group $\N$ as follows:  $n\cdot x := T^{n}(x)$ for any $n \in \N$, where $T^0 = \mathrm{Id}$.  Then for almost every point $x \in E$, the set 
 \begin{equation*}
 \{n \in \N : n\cdot x \in E\}
 \end{equation*} 
 
 \noindent is infinite.  (In other words, the set of points in $E$ that return to $E$ infinitely often has full measure in $E$.) 
 \end{theorem}
 
 Given the additional hypothesis of ergodicity, we can strengthen Poincar\'e recurrence. If $T: X \to X$ is ergodic, then for any measurable set $E \in \mathscr{B}$ almost every $x \in X$ will land in $E$ infinitely often. In other words, $T^{n}(x) \in E$ for infinitely many $n \in \N$. And, in fact, we know how often the point returns. As $n \to \infty$, the ratio of $x$ landing in $E$ and $x$ landing outside of $E$ converges to the measure of the set $E$. However, we can not deduce any quantitative information about the density of the orbits. If we were interested in such information, we could ask the following: given a measurable set $E$, how quickly can we shrink the set $E$ (shrink the set for each application of the transformation $T$) and still have almost every $x \in X$ land in the shrinking sequence of sets infinitely often? More concretely, assume $X$ is a metric space, let $y \in X$, and let $B_{\phi(n)}(y)$ be a ball centered at $y$ with measure $\phi(n)$, where $\phi: \Z_{\geq 0} \to \R_{>0}$ is a decreasing function. How quickly can we decrease the function $\phi$ so that $T^n(x) \in B_{\phi(n)}(y)$ infinitely often for almost every $x \in X$?
 
 \begin{figure}[h]
  \centering
  \captionsetup{justification=centering}
  
  \def\svgwidth{\columnwidth}
%    \rotatebox{0}{\scalebox{0.5}{\input{ShrinkingTarget2.pdf_tex}}}
    \includegraphics[trim=0 100 0 0,clip,width=0.4\textwidth]{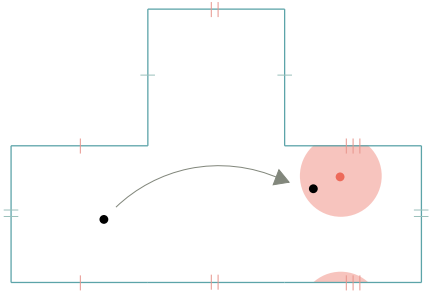}
  \caption{Hitting the target}
  \label{fig:targ2} 
 \end{figure}
 
 Historically, the key to solving such shrinking target problems has been to use the Borel-Cantelli lemma and its partial converse. 
 
 \begin{lemma}[Borel-Cantelli lemma and partial converse]\label{lem:bc}
 Let $(X,\mathscr{B},\mu)$ be a probability space and let $E_n$ be a sequence of measurable sets. 
 
 \begin{enumerate}
 \item (Borel-Cantelli lemma) If $\sum_{n} \mu(E_n) < \infty$, then the set of points $x \in X$ such that $x$ occurs infinitely often has measure $0$ ($\limsup_{n \to \infty} E_n$ has measure 0). 
 \item Conversely, if the $E_n$ are pairwise independent, and $\sum_{n} \mu(E_n) = \infty$, then the set $x \in X$ such that $x$ occurs infinitely often has full measure ($\limsup_{n \to \infty} E_n$ has full measure).  
 \end{enumerate}
 \end{lemma}
 
 To see how the lemma helps us solve a shrinking target problem, consider the following. If $T^n(x)$ lands in the target $B_{\phi(n)}(y)$, then $T^{-n}(B_{\phi(n)}(y))$ must contain $x$. See Figure \ref{fig:will-hit-targ}. 
 
 \begin{figure}[h]
  \centering
  \captionsetup{justification=centering}
  
  \def\svgwidth{\columnwidth}
%    \rotatebox{0}{\scalebox{0.5}{\input{ShrinkingTarget2.pdf_tex}}}
    \includegraphics[trim=0 100 0 0,clip,width=0.41\textwidth]{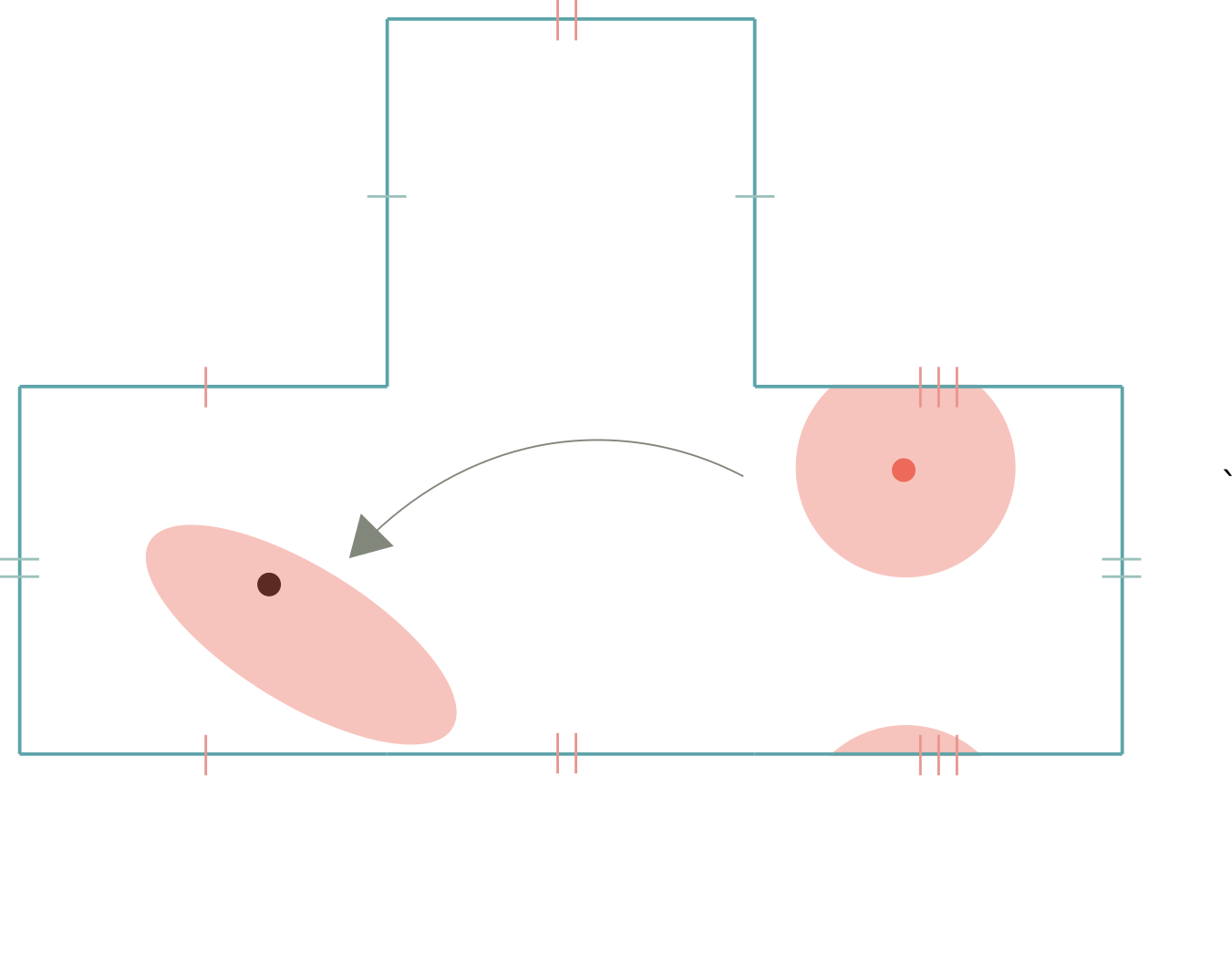}
  \caption{Will hit the target (Application of $T^{-1}$ to target)}
  \label{fig:will-hit-targ} 
 \end{figure}
 
 Now consider the following sum, assuming $T$ is measure-preserving: 
 
 \begin{equation*}
     \sum_{n=0}^{\infty} \mu(T^{-n}(B_{\phi(n)}(y))) = \sum_{n=0}^{\infty} \mu(B_{\phi(n)}(y)) \text{.}
 \end{equation*}
 
 \noindent The first part of Lemma \ref{lem:bc} tells us that if this sum converges, then the set of $x \in X$ such that $x \in T^{-n}(B_{\phi(n)}(y))$ infinitely often has measure zero. In other words, for almost every $x \in X$, there are at most \emph{finitely many} $n \in \N$ such that $T^{n}(x) \in B_{\phi(n)}(y)$. 
 
 Similarly, if the sum above diverges, \emph{and} we have that the sets $B_{\phi(n)}(y)$ are pairwise independent, then we can conclude that the set of points $x \in X$ such that $x \in T^{-n}(B_{\phi(n)}(y))$ infinitely often has full measure. In other words, for almost every $x \in X$, $T^{n}(x) \in B_{\phi(n)}(y)$ for \emph{infinitely many} $n \in \N$. 
 
 By observing convergence or divergence of this sum, we can determine how fast $\phi(n)$ can decrease, or rather, how fast we can shrink the target. But, there is a catch. We often cannot say much regarding \emph{pairwise independence} of the sets. In fact, this property is often absent, which is what makes a shrinking problem both interesting and challenging. We must look for a way to replace this hypothesis. 
 
 \subsection{A Brief History}\label{subsec:st-history}
 
 In 1966, Philipp used the Borel-Cantelli lemma in order to prove certain Diophantine estimates. He did this by formulating a quantitative version of the Borel-Cantelli lemma and used it to show that not only does the $2x$-map on the circle exhibit a shrinking target property, but so does the continued fraction map and the $\theta$-adic expansion map~\cite{Phi67}.
 
 \begin{theorem}[Quantitative Borel-Cantelli lemma]
 Let $E_n$ be a sequence of measurable sets in an arbitrary probability space $(X, \mu)$.  Denote $A(N,x)$ the number of integers $n \leq N$ such that $x \in E_n$.  Define 
 
 \begin{equation*}
 \phi(N) = \sum_{n \leq N} \mu(E_n)
 \end{equation*}
 
 \noindent Suppose that there exists a convergent series $\sum C_k$ with $C_k \geq 0$ such that for all integers $n > m$ we have 
 
 \begin{equation*}
 \mu(E_n \cap E_m) \leq \mu(E_n)\mu(E_m) + \mu(E_n)C_{n-m} \text{.}
 \end{equation*} 
 
 \noindent Then 
 
 \begin{equation*}
 A(N,x) = \phi(N) + O(\phi^{\frac{1}{2}}(N) \, \mathrm{log}^{\frac{3}{2} + \epsilon}(\phi(N)))
 \end{equation*}
 
 \noindent for any $\epsilon > 0 $, for almost every $x \in X$.
 \end{theorem}
 
 Notice how this quantitative version gives an error estimate so that we can understand just how ``far" from pairwise independence the sequence of measurable sets is allowed to be. 
 
 In 1982, Sullivan used a similar idea to prove a logarithm law that describes the cusp excursions of generic geodesics on noncompact, finite volume hyperbolic spaces~\cite{Sul82}. Sullivan constructed a quantitative Borel-Cantelli lemma by replacing the pairwise independence condition with a geometric condition imposed on shrinking sets in the cusps. In 1995, Hill and Velani coined the term "shrinking target" in their fundamental work on the subject~\cite{HiVe95}. Their work begins with an elegant description of the set-up (which we have expanded to include other formulations of shrinking target questions above), then they study the Hausdorff dimensions of the sets of points that hit a target (a Julia set) infinitely often for certain expanding rational maps on the Riemann sphere. Hill and Velani have also studied an analogous shrinking target problem corresponding to $\Z$-actions of affine linear (not necessarily measure preserving) maps on tori~\cite{HiVe95}.  
 
 Kleinbock and Margulis used a Borel-Cantelli argument to prove a far-reaching result that generalizes Sullivan's logarithm law to noncompact, finite volume locally symmetric spaces~\cite{KlMa99}. They replace the pairwise independence condition with exponential decay of correlations of smooth functions on the space. Athreya and Margulis proved that unipotent flows satisfy an analogous logarithm law, using probabilistic methods, techniques from the geometry of numbers, and the exponential decay of correlations of smooth functions on the space~\cite{AtMa09},~\cite{AtMa17}. 
 
 Following the literature on shrinking targets, we give the following definition.   
 
 \begin{definition}[Borel-Cantelli~\cite{At09},~\cite{Fa06}]\label{def:bc}
 Let $G$ be a group acting by measure-preserving transformations on a probability space $(X, \mathscr{B},\mu)$ and let $\Gamma$ be a subgroup. We say that a sequence of measurable sets $\{E_g\}_{g \in \Gamma}$ is \emph{Borel-Cantelli, (BC),} if $\sum_{g \in \Gamma} \mu(E_g) = \infty$ and 
 
 \begin{equation*}
 \mu(\{ x \in X: gx \in E_g \text{ infinitely often} \}) = 1 \text{.}
 \end{equation*}
 \end{definition}
 
 Theorem \ref{thm:3} and Theorem \ref{thm:4} in Section \ref{sec:3-STS} identify conditions for our target sets to be BC. 
 
 For the interested reader, Athreya has provided an expository article on the relationship between shrinking targets, logarithm laws, and Diophantine estimates. The article also speaks to how shrinking target properties can manifest in the various types of dynamical systems~\cite{At09}.

 %%%%%%%%%%%%%%%%%%%%%%%%%%%%%%%
 %%%%%%%%%%%%%%%%%%%%%%%%%%%%%%%
 %%%%%%%%%%%%%%%%%%%%%%%%%%%%%%%
 %%%%%%%%%%%%%%%%%%%%%%%%%%%%%%%
 %% PROOF OF THEOREM %%%%%%%%%%%
 %%%%%%%%%%%%%%%%%%%%%%%%%%%%%%%
 %%%%%%%%%%%%%%%%%%%%%%%%%%%%%%%
 %%%%%%%%%%%%%%%%%%%%%%%%%%%%%%%
 %%%%%%%%%%%%%%%%%%%%%%%%%%%%%%%

 \section{Application to Square-tiled Surfaces}\label{sec:3-STS}

 The goal of this section is to prove Theorem \ref{thm:1}. We begin by reviewing known properties of the Veech group of square-tiled surfaces, and then introduce a representation of the Veech group of the torus. The main contribution is in Section \ref{ssec:cover}: we can use a geometric (covering) argument to yield spectral estimates for the representation of the Veech group of a square-tiled surface that can be used to run a Borel-Cantelli argument. Further, we will show that the results extend to parallelogram-tiled surfaces. 

 \subsection{Properties of the Veech group}
 
 Let $(X,\omega)$ be a lattice surface, and $SL(X,\omega)$ its Veech group. $SL(X,\omega)$ is a non-cocompact lattice subgroup of $SL_2(\R)$, which implies the following: the group contains a hyperbolic element, hence the action of the Veech group on $(X,\omega)$ is ergodic. For the interested reader, proofs of these statements can be found or constructed from the following sources:~\cite{HubSch04},~\cite{K92}.
 
 We will use the following fact, one direction of which was originally proven by Veech~\cite{Vee89}. The equivalence was proven by Gutkin and Judge~\cite{GuJu96},~\cite{GuJu00}.

 Recall that we say two subgroups $\Gamma_1$ and $\Gamma_2$ of $SL_2(\R)$ are \emph{commensurate} if $\Gamma_1 \cap \Gamma_2$ has finite index in both $\Gamma_1$ and $\Gamma_2$. We say that two subgroups $\Gamma_1$ and $\Gamma_2$ of $SL_2(\R)$ are \emph{commensurable} if $\Gamma_1$ is commensurate to a conjugate of $\Gamma_2$. 
 
 \begin{theorem}\label{thm:GJ}
 $(X,\omega)$ is a square-tiled surface if and only if $SL(X,\omega)$ is commensurate with $SL_2(\Z)$. Similarly, $(X,\omega)$ is a parallelogram-tiled surface if and only if $SL(X,\omega)$ is commensurable with $SL_2(\Z)$. 
 \end{theorem}

 We will only need one direction of this statement (the one observed by Veech): a square-tiled surface has a Veech group that is commensurate with $SL_2(\Z)$.

 \subsection{Spectral estimates of the Koopman representation}
 
 In this section, we provide background for Theorem \ref{thm:spectral} below. Theorem \ref{thm:spectral} gives a description of the exponential decay of averages for the action of any convex cocompact subgroup of the Veech group $SL_2(\Z)$ on a torus. It was proven by Finkelshtein in~\cite{Fin16}. 
 
 Let $(X,\mathscr{B},\mu)$ be a probability space, let $\mathcal{U}(L^2(X,\mu))$ be the space of unitary operators, and let $\Gamma$ be a group acting by measure preserving transformations. The \emph{Koopman representation} is the representation $\pi: \Gamma \to \mathcal{U}(L^2(X,\mu))$ defined by $\pi(g)f(x) = f(g^{-1}x)$. Since constant functions are invariant, we will consider the projection $L^2(X,\mu) \to L_0^2(X,\mu)$ where $L_0^2(X,\mu)$ is the closed subspace of functions orthogonal to the constant functions. We denote by $\pi_0$ the representation $\pi_0: \Gamma \to U(L_0^2(X,\mu))$. 
 
 Consider the action of a convex cocompact subroup of $SL_2(\Z)$ on the hyperbolic plane (by M\"obius transformations) and recall the definition of the critical exponent, Definition \ref{critexp}: 
 
 \begin{equation*}
 \delta_{\Gamma} := \limsup_{R \rightarrow \infty} \frac{\mathrm{log}(\#\{g \in \Gamma: d_{\Hh}(g.x_0,x_0) \leq R \})}{R} \text{.}
 \end{equation*}

 %The critical exponent $\delta_{\Gamma}$ is the exponent required for convergence in the Poincar\'e series of the group $\Gamma$~\cite{Bea68}~\cite{Pat76}, which is equivalent to the exponential growth rate of the number of points in the orbit of $\Gamma$ acting on the upper half-plane~\cite{Sul79} seen in  Definition \ref{critexp}.
 
 Patterson showed that the critical exponent for a finitely generated Fuchsian group $\Gamma$ is the Hausdorff dimension of the limit set, $\Lambda = \overline{\Gamma x} \cap S^1$, where $S^1$ is the circle at infinity~\cite{Pat76}. Sullivan showed that the critical exponent of any Fuchsian group is the Hausdorff dimension of the \emph{radial} limit set, $\Lambda_r \subset \Lambda$, which consists of all points in the limit set such that there exists a sequence $\lambda_n x \rightarrow y$ remaining within a bounded distance of a geodesic ray ending at $y$~\cite{Sul79}. Convex cocompact Fuchsian groups are precisely those groups whose limit set \emph{is} the radial limit set. Sullivan studied the radial limit set using the ``density at infinity" (the Patterson-Sullivan measure class) and used this to equate the logarithmic growth rate of the number of orbit points in a ball of radius $R$ of any convex cocompact Fuchsian group acting on $\Hh$ to the Hausdorff dimension of the radial limit set.  The Patterson-Sullivan measure class is a measure class on the boundary of hyperbolic space. The measures themselves are not invariant under the action of the convex cocompact subgroup. Rather, the measure class is invariant. For background on the Patterson-Sullivan measure class and quasiconformal measures, see the survey~\cite{Q}.
 
 In~\cite{C93}, Coornaert generalized the work of Sullivan to the action of a discrete group of isometries on a hyperbolic geodesic metric space by extending the notion of quasiconformal measures to this setting. For background on Gromov hyperbolic spaces, the reader is encouraged to visit~\cite{BH99}. Let $\Gamma$ be a convex cocompact subgroup of $SL_2(\Z)$. If we pick a basepoint $z_0 \in \Hh$, we can define $d_{\Gamma}(g_1, g_2) = d_{\Hh}(g_1 z_0, g_2 z_0)$, and $(\Gamma, d_{\Gamma})$ is an example of a Gromov hyperbolic space. However, $(\Gamma, d_{\Gamma})$ is not a \emph{geodesic} metric space. This poses a problem in applying Coornaert's extension of Patterson-Sullivan theory to the action of $\Gamma$ on $(\Gamma, d_{\Gamma})$, or to the induced action on the Gromov boundary of the group. 
 
 Blach\`ere, Ha\"issinsky, and Mathieu solve this issue by proposing the following coarse characterization of hyperbolicity in~\cite{BHM11}. In short, they define a quasiruled hyperbolic metric space by equipping a  the space with a visual quasiruling structure. With the assumption additional assumption that the metric space is proper, they generalize the Patterson-Sullivan measure class to the non-geodesic setting. In what follows, by quasiconformal measures, we mean a measure in this measure class on the visual boundary (defined below). See Section \S 2 in~\cite{BHM11} for details. An example of a hyperbolic group with a visual quasiruling structure is a convex cocompact $\Gamma \subset \mathrm{Isom}(\Hh)$ with the metric $d_{\Gamma}$.  
 
 In the remainder of this section, we will specialize to $(\Gamma, d_{\Gamma})$, introduce the visual boundary $\partial \Gamma$ of $\Gamma$, and then state the lemma of the shadow according to Blach\`ere, Ha\"issinsky, and Mathieu. As a consequence of the lemma of the shadow, we will define metric balls in $(\Gamma, d_{\Gamma})$ that satisfy certain asymptotics as the radius of the ball increases. Moreover, we will observe that there is a bound on the number of shadows that can cover any element of the boundary provided we pick our shadows carefully. Lastly, we will state the theorem Finkelshtein proved. We encourage the interested reader to consult~\cite{BHM11} and~\cite{Fin16} for details. 

 Let $(\Gamma, d_{\Gamma})$ be as above. 

 \begin{definition}[Gromov product]
 Let $g_0, g, h \in \Gamma$. The Gromov product is $$\left(g,h\right)_{g_0} = \frac{1}{2}\left( d_{\Gamma}(g, g_0) + d_{\Gamma}(h,g_0) - d_{\Gamma}(g,h) \right)\text{.}$$
 \end{definition} 

 We say that a sequence $(g_i)_{i=1}^{\infty} \subset \Gamma$ is a Gromov sequence if $\left(g_i,g_j \right)_{g_0} \to \infty$ as $\min{i,j} \to \infty$. We say that two Gromov sequences $(g_n)$ and $(h_n)$ are equivalent if $\left(g_i , h_i \right)_{g_0} \to \infty$ as $i \to \infty$. We denote an equivalence class of Gromov sequences by $[g_i]$. The visual boundary $\partial \Gamma$ of $(\Gamma, d_{\Gamma}, g_0)$ is the set of equivalence classes of Gromov sequences: $$\partial \Gamma = \left\{ (g_i)_{i=1}^{\infty} : g_i \in \Gamma \text{ and } \lim_{i,j} (g_i, g_j)_{g_0} = \infty \right\} \Big/ \sim \text{.}$$ 

 Moreover, as introduced in~\cite{Sul79} and extended in~\cite{C93}, we have a notion of shadows. Note that some authors define shadows inclusive of elements in $\Gamma$. The shadows we need only include elements in the visual boundary $\partial \Gamma$. Given an element $g \in \Gamma$ and a number $R>0$, the shadow $S_{g_0}(g, R)$ is $$S_{g_0}(g, R) := \left\{ [g_i] \in \partial \Gamma : \liminf_{j \to \infty} (g, g_i)_{g_0} > d_{\Gamma}(g_0, g) - R\right\}\text{.}$$

 Sullivan made a fundamental observation about the measure of the shadows, see \S 2 in~\cite{Sul79}, now called the lemma of the shadow. Coornaert generalized the lemma of the shadow to the action of a discrete group of isometries acting on a hyperbolic geodesic metric space~\cite{C93}, and subsequently Blach\`ere, Ha\"issinsky, and Mathieu generalized this lemma to the setting we are considering~\cite{BHM11}.

 \begin{lemma}[Lemma of the shadow~\cite{BHM11}]\label{lem:shadow}
 Let $\rho$ be a quasiconformal measure based at $g_0$ with respect to the metric $d_{\Gamma}$ on $\Gamma$. There exists $R_0 \geq 0$ such that if $R> R_0$, for any $g \in \Gamma$

 \begin{equation*}
 \rho(S_{g_0}(g,R)) = e^{-\delta_{\Gamma}d_{\Gamma}(g,g_0) + O(1)}\text{.}
 \end{equation*}
 \end{lemma}

 A key consequence of the lemma of the shadow is that we can count orbits in an expanding ball. We record Coornaert's observation about the asymptotics of such a group here. 

 \begin{lemma}\label{lem:shells}\cite{C93}
 Let $\Gamma \subset \mathrm{Isom}(\Hh)$ be convex cocompact and fix a basepoint $z_0 \in \Hh$. Let $\delta_{\Gamma}$ denote the critical exponent of $\Gamma$. Define $B_{n} \subset \Gamma$ to be the set of elements such that $d_{\Hh}(gz_0,z_0) \leq n$, where $gz_0$ denotes the action by M\"obius transformation. Then we have the following asymptotics for the number of elements in $B_n$:
 
 \begin{equation*}
 \lvert B_n \rvert = e^{\delta_{\Gamma}n + O(1)}\text{.}
 \end{equation*}
  
 \noindent Consequently, if we let $S_{n,k} = B_n \setminus B_{n-k}$, then for fixed $k$, we have:

 \begin{equation*}
 \lvert S_{n,k} \rvert = e^{\delta_{\Gamma}n + O(1)}\text{.}
 \end{equation*} 
 \end{lemma}

 Furthermore, Coornaert proved the following lemma in the case of a discrete group of convex cocompact isometries acting on a geodesic metric space. The lemma was originally observed by Sullivan in the setting of convex cocompact groups. 
 
 \begin{lemma}\label{lem:shadow-L}\cite{C93} Let $\Gamma$ be a discrete group of isometries acting on a geodesic hyperbolic metric space $(X,d)$. There exists a $R_0 > 0$ and $k \geq 0$ such that for any $R> R_0$ and $n \in N$, $$\bigcup_{g \in S_{n,k}} S_{e}(g,R) \supset \partial X \text{.}$$ Moreover, there exists $L$ depending $R$ and $k$ such that for any $n$ and any $\xi \in \partial X$, $$\# \{g \in S_{n,k}: \xi \in S_{e}(g,R) \} \leq L \text{.}$$
 \end{lemma}

\begin{comment}
 \begin{proof} Let $g, h \in \Gamma$ such that $[g_i] \in \partial\Gamma$ is in the intersection of $S_{g_0}(g,R)$ and $S_{g_0}(h,R)$, and such that $g, h \in S_{n,k}$ for some $n,k$.

 We will show that $d_{\Gamma}(g,h) \leq 4R + 4C(\delta)+1$. 

 Choose a quasigeodesic ray starting at $g_0$ and terminating at $[g_i] \in \partial \Gamma$. By definition of the shadow $S_{g_0}(g,R)$, there exists an element $p_1$ on the quasigeodesic ray such that $d_{\Gamma}(p_1, g) < R + C(\delta)$, where $C(\delta$ is a constant depending only on the hyperbolicity constant $\delta$ of the metric space $\Gamma$.
 \end{proof}
 \end{comment}

 Finkelshtein recognized that Lemma $\ref{lem:shadow-L}$ holds for the case when the hyperbolic metric space is $(\Gamma, d_{\Gamma})$. Using the work of Blach\`ere, Ha\"issinsky, and Mathieu (specifically, quasigeodesic rays), proofs of the above fact can be translated into the case where a non-elementary group $\Gamma$ acts properly discontinuously and cocompactly by isometries on a proper quasiruled hyperbolic space (such as $\Gamma$ acting on $(\Gamma, d_{\Gamma})$).
 
 Moreover, Finkelshtein observed that one can fix the $k$ in $S_{n,k}$, for all $n$, such that we have both a bound $L$ on the number of overlaps of shadows for all elements in $S_{n,k}$ and the asymptotics in Lemma \ref{lem:shells} hold as $n \to \infty$. He defines the \emph{shell} $S_n \subset \Gamma$ as the set $S_{n,k}$ such that $k$ has this property. Further, he picks $R$ sufficiently large so that all of the shadows he uses in his argument satisfy both the lemma of the shadow, Lemma \ref{lem:shadow}, and Lemma \ref{lem:shadow-L}. With this definition for the shells and by fixing the $R$ parameter in the shadows, Finkelshtein proved the following theorem by studying the induced action of $\Gamma$ on its boundary. The well-behaved harmonic analysis on the torus enabled Finkelshtein to pass from the boundary representation back to the Koopman representation on the torus:
 
 \begin{theorem}\label{thm:spectral}\cite{Fin16}
 Let $\T^2$ be a square torus, let $\Gamma \subset SL_2(\Z)$ be a convex cocompact subgroup with critical exponent $\delta_{\Gamma}$, and let $\pi_0$ denote the Koopman representation on $L_0^2(\T^2)$. Let $\mu_n$ be a uniform probability measure on $S_n \subset \Gamma$. Then
 
 \begin{equation*}
     \lvert \lvert \pi_0(\mu_n) \rvert \rvert \leq e^{-\frac{1}{2} \delta_{\Gamma}n + 2\log n + O(1)} \text{,}
 \end{equation*}
 
 \noindent where $\pi_0(\mu_n) = \sum_{g \in \supp{\mu_n}} \mu_n(g)\pi_0(g)$.
 \end{theorem}

% \textcolor{red}{Needed? Given a representation of a convex cocompact subgroup of $SL_2(\Z)$, we have an induced representation on the visual boundary of the group. This action is not necessarily measure-preserving, but it does preserve a measure class (it is a \emph{quasi-regular representation}). See, for example, \S3 in~\cite{Fin16} and Appendix A in~\cite{BHM11}. For our purposes, defining the Koopman representation will be sufficient, but note that the proof of Theorem \ref{thm:spectral} uses an induced boundary representation.}
 
 This result was used to solve a similar shrinking target problem on a torus. In our set-up, the spectral estimates of convex cocompact subgroups of $SL_2(\Z)$ on the torus play the role of the pairwise independence assumption in the Borel-Cantelli lemma. 

 We require one more property of the shells for the Borel-Cantelli argument in subsection \ref{ssec:bc-argument}: 

 \begin{lemma}\label{lem:shell-norm-bound}
 With $S_n$ as above, and letting $\lvert \lvert \cdot \rvert \rvert$ denote the operator norm of $g$ as a linear transformation on $\R^2$, we have
 
 \begin{equation*}
    \max \{\lvert \lvert g \rvert \rvert: \, g \in S_{2n} \} \leq e^n \text{.}
 \end{equation*}
 
 \end{lemma}
 
 \begin{proof}
 Since $g$ is an isometry of $\Hh$ which lives in a convex cocompact subgroup, it must be either elliptic or hyperbolic. If elliptic, the operator norm is $1$ and we are finished. If hyperbolic, there is a $\delta<e^{2n}$ such that the translation distance of $g$ is $\delta$. Thus $g$ is conjugate in $SL_2(\R)$ to a matrix of the form $\begin{bmatrix} e^{\frac{\delta}{2}} & 0 \\ 0 & e^{-\frac{\delta}{2}} \end{bmatrix}$. It follows that the eigenvalues are $e^{\frac{\delta}{2}}$ and $e^{-\frac{\delta}{2}}$ and we conclude that $\lvert \lvert g \rvert \rvert$ is $e^{\frac{\delta}{2}} < e^{n}$. Since the largest eigenvalue of a hyperbolic matrix is the operator norm, we are done. 
 \end{proof}
 
 \subsection{A covering argument}\label{ssec:cover}

 \begin{lemma}\label{lem:equivariance}
 Let $(X,\omega)$ be a square-tiled surface and let $SL(X,\omega)$ be its Veech group. Then, there exist a finite index subgroup $\Gamma^{\prime} \subset SL(X,\omega)$ and a branched cover $q: X \to \T^2$ such that the cover is equivariant with respect to the action of any subgroup $\Gamma^{\prime}$. 
 \end{lemma}

 This follows from Theorem \ref{thm:GJ}, and the work in~\cite{GuJu00}. A similar statement holds for parallelogram-tiled surfaces, but we need to compose a cover of a (non-square) torus with an affine map to a square torus. 
 
 Although irrelevant to our current goals, it is worth noting that we know precisely which square-tiled surfaces (and parallelogram-tiled surfaces) for which the full Veech group descends to an action on the square torus (or to an action conjugate to an action on the square torus). Recall that the \emph{saddle connections} of a translation surface $(Y,\nu)$ are straight line trajectories that start at a singular point and end at a singular point, passing through no singular points in-between. The \emph{holonomy vectors} are the values we get when we integrate the saddle connections over the holomorphic one-form $\nu$. The \emph{period lattice} is the the lattice generated by the holonomy vectors.
 
 For a square-tiled surface $(X,\omega)$ (tiled by unit squares), the holonomy vectors must be a subset of $\Z \oplus i\Z$. We say a square-tiled surface is \emph{reduced} if the period lattice is $\Z \oplus i\Z$. Since the action of the Veech group preserves the period lattice, we see that the Veech group must be a subgroup of $SL_2(\Z)$, and further, that for any $g \in SL(X,\omega)$, the following diagram commutes
 
 \begin{equation*}
    \begin{tikzcd}
        X \ar[r, "g"] \ar[d, "q"] & X \ar[d, "q"]\\
        \T^2  \ar[r, "g"]  & \T^2 \\
    \end{tikzcd}
 \end{equation*}
 
 \noindent where $q(x) = \int_p^x \omega \mod \Z \oplus i\Z$. (Note that this map respects the choice of north on the square-tiled surface.)
 
 If the period lattice is not $\Z \oplus i\Z$, then there will be a parabolic element in the Veech group with non-integral entries. Such an element cannot descend with respect to the cover.  
 
 There is a similar picture for parallelogram-tiled surfaces, where the tiling parallelogram has sides $a,b \in \C$ and unit area. Let $P$ denote the translation surface given by identifying opposite sides of the parallelogram. Then, we say that $(X,\omega)$ is a \emph{reduced} parallelogram-tiled surface if the period lattice is $a\Z \oplus b\Z$. For a reduced parallelogram-tiled surface, the following diagram commutes.  
 
  \begin{equation*}
    \begin{tikzcd}
        X \ar[r, "g"] \ar[d, "q"] & X \ar[d, "q"]\\
        P  \ar[r, "g"] \ar[d, "h"]  & P \ar[d, "h"]\\
        \T^2  & \T^2 \\
    \end{tikzcd}
 \end{equation*}
 
 \noindent where $q(x) = \int_p^x \omega \mod \Z[a] \oplus \Z[b]$, $h \in SL_2(\R)$, and as above, the choice of north on the translation surfaces is respected with the exception of the action of $h$. 

 We now turn our attention to functions on the space. Let $(X,\omega)$ be a square-tiled surface with probability measure $\nu$ and let $q: X \to \T^2$ be the branched covering map. Then 
 
 \begin{equation*}
     L^2(X, \nu) \cong H \oplus H^{\perp} 
 \end{equation*}
 
 \noindent where $H$ is the pullback of $L^2(\T^2)$.
 
 \begin{lemma}\label{lem:contravariance}
 Let $q^*: L^2(\T^2) \rightarrow H$ denote the pullback. There exists a finite index subgroup $\Gamma^{\prime} \subset SL(X,\omega)$ such that for every $g \in \Gamma^{\prime}$, the following diagram commutes.
 
  \begin{equation*}
    \begin{tikzcd}
        H \ar[r, "\pi_H(g)"] & H \\
        \mathrm{L^2}(\T^2) \ar[u, "q^*"] \ar[r, "\pi(g)"]  & \mathrm{L^2}(\T^2) \ar[u, "q^*"] \\
    \end{tikzcd}
 \end{equation*}
 
 \noindent where $\pi_H: \Gamma^{\prime} \to U(H)$ is the Koopman representation of $SL(X,\omega)$ on $H$, and $\pi: \Gamma^{\prime} \to U(L^2(\T^2))$ is the Koopman representation of $\Gamma^{\prime}$ on $L^2(\T^2)$.
 
 \end{lemma}
 
 \begin{proof}
 This follows from the equivariance of the covering map, Lemma \ref{lem:equivariance}. In fact, the group representation $\pi_H$ is well-defined because of Lemma \ref{lem:equivariance}.
 \end{proof}
 
 \begin{corollary} 
 With hypothesis as in Lemma \ref{lem:contravariance}, for every $g \in \Gamma^{\prime}$,  
 
 \begin{equation*}
     \lvert \lvert \pi_H(g) \rvert \rvert = \lvert \lvert \pi(g) \rvert \rvert \text{.}
 \end{equation*}
 \end{corollary}
 
 As a consequence, and by applying Theorem \ref{thm:spectral}, we have the following. 
 
 \begin{corollary}\label{cor:spectral}
 Let $H_0 \subset H$ such that $H_0 = q^*(L_0^2(\T^2))$, the subspace of $H$ orthogonal to the constant functions. Let $\pi_{H_0}: \Gamma^{\prime} \to U(H_0)$, which is well-defined since the space of constant functions is invariant under the representation $\pi_H$ defined above. Let $\mu_n$ be the measure from Theorem \ref{thm:spectral}. Then
 
 \begin{equation*}
    \lvert \lvert \pi_{H_0}(\mu_n) \rvert \rvert = \lvert \lvert \pi_0(\mu_n) \rvert \rvert \leq  e^{-\frac{1}{2} \delta_{\Gamma}n + 2\log n + O(1)} \text{.}
 \end{equation*}
 \end{corollary}
 
 The technique above provides a framework for lifting spectral estimates using a cover, so estimates on ``primitive" translation surfaces, those that do not cover (with finite branching) other translation surfaces, can be lifted to surfaces covered by the primitive surface. 
 
 \subsection{Borel-Cantelli argument}\label{ssec:bc-argument}
 
 In this section, we show how to use the spectral estimates to run a Borel-Cantelli argument similar to~\cite{Fin16}, but with variations to accommodate a reduction to a finite index subgroup and the tiling of the surface. 
 
 \begin{theorem}\label{thm:3}
 Let $(X,\omega)$ be a square-tiled (or parallelogram-tiled) surface with $M$ squares in the tiling and let $q: X \to \T$ be the branched cover over the torus. Let $\mu$ denote the normalized Lebesgue measure so that $\mu(X)=1$ and let $\Gamma \subset SL(X,\omega)$ be a convex cocompact subgroup with critical exponent $\delta_{\Gamma} > 0$. Define a sequence of measurable sets $\mathrm{Targ}_{\phi(r)}$ of measure $\phi(r)$ where $\phi: \R_{>0} \to \R_{>0}$ is a non-increasing function such that if $r_1 > r_2$, we have $\mathrm{Targ}_{\phi(r_1)} \subset \mathrm{Targ}_{\phi(r_2)}$. Moreover, we require that there exists an $R$ such that for all $r>R$, the set $\mathrm{Targ}_{\phi(r)}$ is saturated with respect to the cover, meaning $q^{-1} \circ q(\mathrm{Targ}_{\phi(x)}) = \mathrm{Targ}_{\phi(x)}$. Then for almost every $x \in X$, the set 
 
 \begin{equation*}
     \left\{g \in \Gamma : gx \in \mathrm{Targ}_{\phi(\lvert \lvert g \rvert \rvert)}\right\}
 \end{equation*}
 
 \noindent is 
 
 \begin{enumerate}
     \item finite, if $\sum_{n=1}^{\infty} n^{2\delta_{\Gamma} - 1} \phi(n) < \infty$.
     \item infinite, if $\sum_{n=1}^{\infty} n^{-(2\delta_{\Gamma} + 1)} \log^4(n) \phi(n)^{-1} < \infty$. 
 \end{enumerate}
 
 \end{theorem}

 \begin{remark} There is a gap between the finite and infinite case. If $\phi(r) = C r^{-2\delta_{\Gamma}}\log^{1+\varepsilon}(r)$ for all large $r$, where $C$ is any constant and $\varepsilon>0$, there will only be finitely many solutions. If $\phi(r) = C r^{-2\delta_{\Gamma}} \log^{5+\varepsilon}(r)$, for all large $r$, where again $C$ is any constant and $\varepsilon>0$, there will be infinitely many solutions. It may be possible to strengthen this theorem. 
 \end{remark}

% \begin{remark} The assumption regarding the existence of $N$ such that for all $x>N$, the set $q(\mathrm{Targ}_{\phi(x)})$ is 
% \end{remark}

 \begin{proof}
 Fix a basepoint $x_0$, and let $S_n$ be as in Theorem \ref{thm:spectral} where $S_n = S_{n,k}$ for some $k > 0$. In our proof, we will not impose any constraints on the value of $k$, but we will eventually use the conclusion of Theorem \ref{thm:spectral}.
 
 First, we show that the sum converges under the first condition, which by Borel-Cantelli (Lemma \ref{lem:bc}) implies that the set is finite.  Let $B_n = \{ g \in \Gamma : d_{\Hh}(gz_0,z_0) \leq n \}$ (as in Lemma \ref{lem:shells}). Observe that   

 \begin{enumerate} 
 \item $B_k \cup \left(\bigcup_{n = k+1}^{\infty} B_{n} \setminus B_{n-1} \right) = \Gamma$, 
 \item $(B_{n}\setminus B_{n-1}) \bigcap (B_{m}\setminus B_{m-1}) = \emptyset$ if and only if $n \neq m$ and $n, m > 0$, and
 \item $B_k \cap B_n\setminus B_{n-1} = \emptyset$ for all $n > k$.
 \end{enumerate}

% Recall Lemma \ref{lem:shell-norm-bound}: for $g \in S_{2n}$, we have that $\lvert \lvert g \rvert \rvert \leq e^n$. This implies that 

% \begin{enumerate}
% \item $S_{2n} \setminus S_{2n-2} = \left\{ g \in \Gamma : e^{n-1} < \lvert \lvert g \rvert \rvert \leq e^n \right\}$, 
% \item $\bigcup_{n = 1}^{\infty} S_{2n} \setminus S_{2n-2} = \Gamma$, where $S_0 = \emptyset$, and 
% \item $(S_{2n}\setminus S_{2n-2}) \bigcap (S_{2m}\setminus S_{2m-2}) = \emptyset$ if and only if $n \neq m$.
% \end{enumerate}

% First, we show that the sum converges under the first condition, which by Borel-Cantelli (Lemma \ref{lem:bc}) implies that the set is finite. Let $\Gamma_n = \left\{ g \in \Gamma : e^{n-1} < \lvert \lvert g \rvert \rvert \leq e^n \right\}$. Then, $$\Gamma = \cup_{n=0}^{\infty} \Gamma_n$$ and $$\Gamma_{n_1} \cap \Gamma_{n_2} = \emptyset$$ if and only if $$n \neq m$$.
 
 \noindent It follows that
 
 \begin{align*}
     \sum_{g \in \Gamma} \mu(g^{-1}\mathrm{Targ}_{\phi(\lvert \lvert g \rvert \rvert)}) &= \sum_{n=k+1}^{\infty} \sum_{g \in B_{n} \setminus B_{n-1}} \phi(\lvert \lvert g \rvert \rvert) + \sum_{g \in B_k} \phi(\lvert \lvert g \rvert \rvert),\\
 \end{align*}
 
 \noindent where the second term on the right-hand side is a sum over finitely many elements. We focus on the tail of the series. Let $N$ be sufficiently large and use the precise asymptotics of the balls, Lemma \ref{lem:shells}. Moreover, observe that Lemma \ref{lem:shell-norm-bound} implies that for any $g \in B_n$, $\lvert \lvert g \rvert \rvert \leq e^{\frac{n}{2}}$. These observations yield 
 
 \begin{align*}
     \sum_{n=N}^{\infty} \sum_{g \in B_{n} \setminus B_{n-1}} \phi(\lvert \lvert g \rvert \rvert) &\leq C_1 \sum_{n=N}^{\infty} e^{\delta_{\Gamma} n} \phi(\lvert \lvert g \rvert \rvert) \\
     &= C_1 \sum_{n=N}^{\infty} e^{\delta_{\Gamma} n}  \phi(e^{\frac{n-1}{2}}) \\
     &= C_1 \sum_{n=N+1}^{\infty} e^{\delta_{\Gamma} n}  \phi(e^{\frac{n}{2}}) \text{.}\\
 \end{align*}
 
 \noindent for some constant $C_1>0$. We apply a variation of the Cauchy condensation test to establish the criteria for convergence. For a non-increasing function $f: \N \to \R_{\geq 0}$, $\sum_{n} f(n)$ converges if and only if $\sum_{n} e^nf(e^n)$ converges. Observe that 

 \begin{align*}
     \sum_{n=N+1}^{\infty} e^{\delta_{\Gamma} n}  \phi(e^{\frac{n}{2}}) &= \sum_{n=N+1}^{\infty} e^{\frac{n}{2}} e^{\frac{n}{2}(2\delta_{\Gamma} -1)}  \phi(e^{\frac{n}{2}}) \text{.}\\
 \end{align*}
 
 \noindent By splitting the sum over the odd and even integers greater than $N + 1$, bounding the sum by twice the sum over the even integers, and re-indexing, we can use Cauchy condensation to deduce that the series above converges if and only if $\sum_{n=1}^{\infty} n^{2\delta_{\Gamma} - 1}\phi(n)$ converges. By Borel-Cantelli, we can conclude that for almost every $x \in X$ the set $$\left\{g \in \Gamma : gx \in \mathrm{Targ}_{\phi(\lvert \lvert g \rvert \rvert)}\right\}$$ is finite. 
 
 For the more difficult part of the proof, we show the set $\left\{g \in \Gamma : gx \in \mathrm{Targ}_{\phi(\lvert \lvert g \rvert \rvert)}\right\}$ is infinite when $\phi$ satisfies the condition in (2). For this part of the proof we will use the assumption that there exists an $R$ such that for every $r > R$, $q^{-1}\circ q (\mathrm{Targ_{\phi(r)}}) = \mathrm{Targ_{\phi(r)}}$ (the target is a saturated set with respect to the cover). %Once we have proven the statement for the saturated set, we will remove the assumption that the target set is saturated, replace it with the assumption in the theorem statement, and show that the theorem holds.  
 
 In this part of the proof, a key step requires applying Theorem \ref{thm:spectral}, so we will use the shells $S_n$. For $n > k$, let 
 
 \begin{equation*}
 E_n =  X \setminus \bigcup_{g \in S_{n}} g^{-1}\mathrm{Targ}_{\phi(\lvert \lvert g \rvert \rvert)}\text{.}
 \end{equation*}
 
 \noindent $E_n$ is the set of points $x$ such that $gx$ is not in its target $\mathrm{Targ}_{\phi(\lvert \lvert g \rvert \rvert)}$ for all $g \in S_{n}$. Let $$E = \limsup_{n \to \infty} E_n \text{.}$$ $E$ consists of the points for which there are infinitely many $n$ such that for all $g \in S_{n}$, $gx$ is not in the target. The complement of this set is the set of $x \in X$ such that there are only finitely many $n$ such that $gx$ misses the target for all $g \in S_{n}$. This is a subset of the points which hit the target infinitely often. If we show that $\mu(E) = 0$, then we will have shown that the set of $x \in X$ which land in the target infinitely often has full measure. 
 
 We begin by reducing to a finite index subgroup $\Gamma^{\prime} \subset \Gamma$ so that we can apply the results of Theorem \ref{thm:spectral}. For finite index subgroups of convex cocompact subgroups we have that $\delta_{\Gamma^{\prime}} = \delta_{\Gamma}$. Moreover, the finite index subgroup is convex cocompact (being finitely generated without parabolic elements). 
 
 Let $\mathrm{Targ}_n$ denote the set $\mathrm{Targ}_{\phi\left(e^{\frac{n}{2}}\right)}$ and let $\chi_{\mathrm{Targ_n}}$ be the characteristic function of that set. Let $\chi_{E_n}$ be the characteristic of the set $E_n$. Let $M$ be the number of squares tiling $(X,\omega)$ and recall that $q: X \to \T^2$ is the covering map. We will define two bounded linear operators. First, define $A: L^2(X) \to H$ by 
 
 \begin{equation*}
     A(f) = \tilde{f}
 \end{equation*}
 
 \noindent where $\tilde{f}(x) = \frac{1}{M} \sum_{y \in q^{-1}\circ q(x)} f(y)$. Note that $\int_X A(f) \,d\mu = \int_X f\, d\mu$.
 
 \noindent Now define the orthogonal projection $P: L^2(X) \to H_0$ by 
 
 \begin{equation*}
     P(f) = A(f) - \int_X f \, d\mu \text{.}
 \end{equation*}
 
 \noindent $P$ is a self-adjoint, idempotent operator. First, $\langle Pf, g \rangle = \langle f, Pg \rangle$:
 
 \begin{align*}
     \langle Pf, g \rangle - \langle f, Pg \rangle &= \int_X \left(A(f) - \int_X f \,d\mu\right) g \, d\mu - \int_X \left(A(g) - \int_X g \,d\mu\right) f \, d\mu\\
     &= \int_X A(f) g \,d\mu -\int_X f\,d\mu \int_X g\,d\mu - \int_X A(g)f \, d\mu + \int_X f \, d\mu \int_X g\,d\mu \\
     &= \int_X A(f) g \,d\mu - \int_X A(g)f \, d\mu \\
     &= 0 \text{,}
 \end{align*}
 
 \noindent where the last line follows from considering the integral over each square. 
 
 \noindent Second, $P^2(f) = P(f)$:
 
 \begin{align*}
     P^2(f) &= P(A(f) - \int_X f \, d\mu) \\
     &= P(A(f) - \int_X A(f) \, d\mu) \\
     &= P(A(f)) - P(\int_X A(f) \, d\mu) \\
     &= A(A(f)) - \int_X A(f)\, d\mu - 0 \\
     &= A(f) - \int_X f \, d\mu \\
     &= P(f)\text{.}
 \end{align*}
 
 We can project the characteristic functions $\chi_{\mathrm{Targ}_n}$ and $\chi_{E_n}$ to $H_0$:
 
 \begin{align*}
     T_n &:= P(\chi_{\mathrm{Targ}_n}) = A(\chi_{\mathrm{Targ}_n}) - \mu(\mathrm{Targ}_n) \\ 
     \mathrm{Bad}_n &:= P(\chi_{E_n}) = A(\chi_{E_n}) - \mu(E_n) \text{.}\\ 
 \end{align*}
 
 \noindent Now observe that 
 
 \begin{align*}
     \lvert \lvert T_n \rvert \rvert^2_2 &\leq (1 - \mu(\mathrm{Targ}_n))\mu(\mathrm{Targ}_n) \leq \mu(\mathrm{Targ}_n)\\
     \lvert \lvert \mathrm{Bad}_n \rvert \rvert^2_2 &\leq (1 - \mu(E_n))\mu(E_n) \leq \mu(E_n) \text{.}\\
 \end{align*}

 \noindent Indeed, consider $\lvert \lvert T_n \rvert \rvert^2_2$. Define sets $$A_i = \left\{ x \in X \, : \, A(\chi_{\mathrm{Targ_n}})(x) = \frac{i}{M} \right\}$$ for $i \in \{0, 1, 2, \cdots, M\}$. Observe that $X = \bigsqcup_{i=0}^M A_i$, where the $A_i$ are disjoint, and $\sum_{i=1}^M \frac{i}{M} \mu(A_i) = \mu(\mathrm{Targ}_n)$:

 \begin{equation*}
 \mu(\mathrm{Targ}_n) = \int_X \chi_{\mathrm{Targ}_n} \, d\mu = \int_X A(\chi_{\mathrm{Targ}_n}) \, d\mu = \int \sum_{i=1}^M \frac{i}{M} \chi_{A_i} \, d\mu = \sum_{i=1}^M \frac{i}{M} \mu(A_i) \text{.}
 \end{equation*}

 \noindent Then, 

 \begin{align*}
 \lvert \lvert T_n \rvert \rvert^2_2 & = \int_X \left\lvert A(\chi_{\mathrm{Targ_n}}) - \mu(\mathrm{Targ_n}) \right\rvert^2 \, d\mu \\
 &= \int_X \left( A(\chi_{\mathrm{Targ}_n}) \right)^2 d\mu \, - \, 2\mu(\mathrm{Targ}_n) \int_X A(\chi_{\mathrm{Targ}_n}) \, d\mu \, + \, \mu(\mathrm{Targ}_n)^2\\
 &= \int_X \left( A(\chi_{\mathrm{Targ}_n}) \right)^2 d\mu \, - \, \mu(\mathrm{Targ}_n)^2\\
 &= \sum_{i=1}^{M} \int_{A_i} \left( \frac{i}{M} \right)^2 d\mu \, - \, \mu(\mathrm{Targ}_n)^2\\
 &= \sum_{i=1}^{M} \left( \frac{i}{M} \right)^2 \mu(A_i) \, - \, \mu(\mathrm{Targ}_n)^2\\
 &\leq \sum_{i=1}^{M} \left( \frac{i}{M} \right) \mu(A_i) \, - \, \mu(\mathrm{Targ}_n)^2\\
 &= \mu(\mathrm{Targ}_n) \, - \, \mu(\mathrm{Targ}_n)^2\\
 &= (1 - \mu(\mathrm{Targ}_n))\mu(\mathrm{Targ}_n) \\
 \end{align*}

 \noindent The computation is similar for $\lvert \lvert \mathrm{Bad}_n \rvert \rvert^2_2$.
 
 \noindent Moreover, 
 
 \begin{align*}
     \langle T_n, \mathrm{Bad}_n \rangle &= \int_X \left( A(\chi_{\mathrm{Targ_n}}) - \mu(\mathrm{Targ_n}) \right) \left( A(\chi_{E_n}) - \mu(E_n) \right) \, d\mu\\
     &= \int_X A(\chi_{\mathrm{Targ_n}}) A(\chi_{E_n}) \, d\mu - \mu(E_n) \int_X A(\chi_{\mathrm{Targ_n}}) \, d\mu \, - \, \mu(\mathrm{Targ}_n)\int_X A(\chi_{E_n})\, d\mu + \mu(\mathrm{Targ}_n) \mu(E_n)\\
     &= \int_X A(\chi_{\mathrm{Targ}_n}) A(\chi_{E_n}) \, d\mu - \mu(\mathrm{Targ}_n) \mu(E_n) \text{.} \\
 \end{align*}

 We will show that provided $n > 2 \log{R}$, where $R$ comes from our Theorem statement, and that $g \in S_{n}$, we have 

 \begin{equation*}
 \left\lvert \langle \pi_{H_0}(g^{-1}) T_n, \mathrm{Bad}_n \rangle \right\rvert = \mu(\mathrm{Targ}_n) \mu(E_n)
 \end{equation*}
 
 \noindent Fix any $g \in S_{n} \subset \Gamma^{\prime}$ and observe that $g^{-1} \in S_{n}$. We have 
 
 \begin{align*}
     \langle \pi_{H_0}(g^{-1}) T_n, \mathrm{Bad}_n \rangle &= \int_X A(\chi_{\mathrm{Targ}_n}(gx)) A(\chi_{E_n}(x)) \, d\mu(x) - \mu(\mathrm{Targ}_n) \mu(E_n) \\
     &= \int_X A(\chi_{g^{-1}\mathrm{Targ}_n}(x)) A(\chi_{E_n}(x)) \, d\mu(x) - \mu(\mathrm{Targ}_n) \mu(E_n) \\
     &= \int_X \left(\frac{1}{M} \sum_{y \in q^{-1}\circ q (x)} \chi_{g^{-1}\mathrm{Targ}_n}(y) \right) \left( \frac{1}{M} \sum_{ y \in q^{-1}\circ q (x)} \chi_{E_n}(y) \right) \, d\mu(x) - \mu(\mathrm{Targ}_n) \mu(E_n) \text{.} \\
 \end{align*}

 \noindent The integrand is zero. To be non-zero at $x \in X$, there must exists two points $y_1, y_2 \in q^{-1} \circ q(x)$ such that $y_1 \in g^{-1}\mathrm{Targ}_n$ and $y_2 \in E_n$. Given our assumption that the target set is saturated with respect to the cover (we pick $n > 2 \log{R}$), this is not possible. If $\mathrm{Targ}_n$ is a saturated set, then so is $g^{-1}\mathrm{Targ}_n$ ($g^{-1}$ is a continuous map that commutes with the cover). If $y_1 \in g^{-1}\mathrm{Targ}_n$, then every $y \in q^{-1} \circ q(x)$ is also in $g^{-1}\mathrm{Targ}_n$. This means that $y_2$ must be in both $g^{-1}\mathrm{Targ}_n$ and $E_n$. However, if

 \begin{equation*}
 g(y_2) \in \mathrm{Targ}_{n} 
 \end{equation*}

 \noindent and 

 \begin{equation*}
 h(y_2) \notin \mathrm{Targ}_{\phi(\lvert \lvert h \rvert \rvert)}
 \end{equation*}

 \noindent for all $h \in S_n$, we get a contradiction. $g(y_2) \in \mathrm{Targ}_n = \mathrm{Targ}_{\phi\left(e^{\frac{n}{2}}\right)} \subset \mathrm{Targ_{\phi(\lvert \lvert g \rvert \rvert)}}$, but $g \in S_{n}$. Thus, we can conclude that for $n > 2 \log(R)$ and $g \in S_n$,

 \begin{equation*}
 \left\lvert \langle \pi_{H_0}(g^{-1}) T_n, \mathrm{Bad}_n \rangle \right\rvert = \mu(\mathrm{Targ}_n) \mu(E_n) \text{.}
 \end{equation*}

 \noindent Moreover, if we let $\mu_n$ be a uniform probability distribution on $S_n$, we can conclude

 \begin{equation*}
 \left\lvert \langle \pi_{H_0}(\mu_n) T_n, \mathrm{Bad}_n \rangle \right\rvert = \mu(\mathrm{Targ}_n) \mu(E_n) \text{.}
 \end{equation*}

 \noindent Using Cauchy-Schwarz on this inner product as well as the bounds on $\lvert \lvert T_n \rvert \rvert_2^2$ and $\lvert \lvert \mathrm{Bad}_n \rvert \rvert_2^2$, we can relate the operator norm of $\lvert \lvert \pi_{H_0}(\mu_{n}) \rvert \rvert$ to the measures $\mu(E_n)$ and $\mu(\mathrm{Targ}_n)$. 
 
 \begin{equation*}
     \left\lvert \langle (\pi_{H_0}(\mu_{n})) T_n, \mathrm{Bad}_n \rangle \right\rvert \leq \left(\lvert \lvert \pi_{H_0}(\mu_{n})\rvert \rvert \right) \, \mu(\mathrm{Targ}_n)^{\frac{1}{2}} \mu(E_n)^{\frac{1}{2}}
 \end{equation*}
 
 \noindent By applying the spectral estimate in Corollary \ref{cor:spectral} in combination with the previous two equations, we can deduce that for some constant $C_2 > 0$ 
 
 \begin{align*}
     \mu(\mathrm{Targ}_n) \mu(E_n) & \leq \left(\lvert \lvert \pi_{H_0}(\mu_{n})\rvert \rvert \right) \, \mu(\mathrm{Targ}_n)^{\frac{1}{2}} \mu(E_n)^{\frac{1}{2}}\\
     \mu(E_n) & \leq \left(\lvert \lvert \pi_{H_0}(\mu_{n})\rvert \rvert \right)^2 \, \mu(\mathrm{Targ}_n)^{-1} \\
     &\leq C_2 n^4 e^{-\delta_{\Gamma}n} \phi(e^\frac{n}{2})^{-1} \text{.} \\
 \end{align*} 
 
 \noindent Thus,  
 
 \begin{equation*}
     \sum_{n> 2\log R}^{\infty} \mu(E_n) \leq C_2 \sum_{n> 2\log R}^{\infty} n^4  e^{-\delta_{\Gamma}n} \phi(e^\frac{n}{2})^{-1}  
 \end{equation*} 
 
 \noindent In order for $\sum_n \mu(E_n)$ to converge, which by the Borel-Cantelli lemma would imply that $\mu(E) = 0$, we need the following sum to converge:

 \begin{equation*}
 \sum_{n> 2 \log R}^{\infty} n^4 e^{-2\delta_{\Gamma}} \phi(e^n)^{-1} \text{.}
 \end{equation*}

 \noindent By using the same variation of the Cauchy condensation test as we used prior, we deduce that convergence the sum is equivalent to convergence of the following sum.
 
 \begin{equation*}
     \sum_{n=1}^{\infty} (\log n)^4 n^{-(2\delta_{\Gamma} + 1)} \phi(n)^{-1}
 \end{equation*}
 
 \noindent This completes the proof for target sets $\mathrm{Targ_{\phi(r)}}$ that are eventually saturated.

 \end{proof}

 If we further assume that the square-tiled surface is a regular square-tiled surface, then we can conclude that each sheet in a saturated target is hit infinitely often. This allows us to remove the assumption of a saturated set. However, to do this, we need to add the assumption that the target sets become sufficiently small in the sense that eventually, the projection of the target is an evenly covered set.

  \begin{theorem}\label{thm:4}
  Let $(X,\omega)$ be a regular square-tiled (or parallelogram-tiled) surface with $M$ squares in the tiling and let $q: X \to \T$ be the branched cover over the torus. Let $\mu$ denote the normalized Lebesgue measure so that $\mu(X)=1$ and let $\Gamma \subset SL(X,\omega)$ be a convex cocompact subgroup with critical exponent $\delta_{\Gamma} > 0$. Define a sequence of measurable sets $\mathrm{Targ}_{\phi(r)}$ of measure $\phi(r)$ where $\phi: \R_{>0} \to \R_{>0}$ is a non-increasing function such that if $r_1 > r_2$, we have $\mathrm{Targ}_{\phi(r_1)} \subset \mathrm{Targ}_{\phi(r_2)}$. Moreover, we require that there exists an $R$ such that for all $r>R$, the set $q(\mathrm{Targ}_{\phi(r)})$ is evenly covered and that one of the sheets is $\mathrm{Targ}_{\phi(r)}$. Then for almost every $x \in X$, the set 
 
 \begin{equation*}
     \left\{g \in \Gamma : gx \in \mathrm{Targ}_{\phi(\lvert \lvert g \rvert \rvert)}\right\}
 \end{equation*}
 
 \noindent is 
 
 \begin{enumerate}
     \item finite, if $\sum_{n=1}^{\infty} n^{2\delta_{\Gamma} - 1} \phi(n) < \infty$.
     \item infinite, if $\sum_{n=1}^{\infty} n^{-(2\delta_{\Gamma} + 1)} \log^4(n) \phi(n)^{-1} < \infty$. 
 \end{enumerate}
 
 \end{theorem}

 \begin{proof} As in the proof of Theorem \ref{thm:3}, the finite case follows from a direct application of Borel-Cantelli, Lemma \ref{lem:bc}. The infinite case follows from Theorem \ref{thm:3}, and the fact that a regular square-tiled surface has a transitive automorphism group. First, observe that Theorem \ref{thm:3} implies $$\left\{g \in \Gamma : gx \in q^{-1}\circ q\left(\mathrm{Targ}_{\phi(\lvert \lvert g \rvert \rvert)}\right)\right\}$$ is infinite when $\phi$ is as in the theorem. Morever, we can restrict to $g \in \Gamma$ such that $\| g \| > R$, and the set $$\left\{g \in \Gamma : \|g\| > R \text{ and } gx \in q^{-1}\circ q\left(\mathrm{Targ}_{\phi(\lvert \lvert g \rvert \rvert)}\right)\right\}$$ remains infinite. Indeed, the proof of Theorem \ref{thm:3} shows that for almost every $x \in X$, there are infinitely many $n$ such that there exists a $g \in S_n$ such that $gx$ hits the target. 
 
 By assumption, for any $\epsilon > 0$, $q(\mathrm{Targ}_{\phi(R+\varepsilon)})$ is evenly covered. Consider the sheets of $q(\mathrm{Targ}_{\phi(R+\varepsilon)})$. Corresponding to each sheet is a sequence of measurable sets contained in that sheet, where each measurable set is a lift of $q(\mathrm{Targ}_{\phi(r)})$ for $r > R$. Since $\mathrm{Targ}_{\phi(r_1)} \subset \mathrm{Targ}_{\phi(r_2)}$ for $r_1 > r_2$, the sequences corresponding to each sheet satisfy the inclusion property. Further, by the pigeonhole principle, one of the sequences is hit infinitely often. If this sequence is $\mathrm{Targ}_{\phi(r)}$, then we are done. If not, call this sequence $U_r$. 
 
 The action of the automorphism group on the squares is transitive and $q(\mathrm{Targ}_{\phi(\|g\|)})$ is evenly covered, so there exists an automorphism $f: X \to X$ such that $\mathrm{Targ}_{\phi(r)}$ is mapped to $U_r$, for all $r$. For almost every $x \in X$, we know that the set $\left\{g \in \Gamma : \|g\| > R \text{ and } f^{-1}(gx) \in f^{-1}(U) = \mathrm{Targ}_{\phi(\lvert \lvert g \rvert \rvert)}\right\}$ is infinite. 
 
 We make two observations. First, recall that the Veech group $SL(X,\omega)$ is the image of the derivative map $D: \mathrm{Aff}(X,\omega) \to SL(X,\omega)$, where $\mathrm{Aff}(X,\omega)$ is the group of affine transformations of the surface. Moreover, the kernel of this map is the set of automorphisms of the translation structure, $\mathrm{Aut}(X,\omega)$, hence the set of automorphisms is a normal subgroup of the group of affine transformations. See, for example, ~\cite{HubSch04}. We can inject $SL(X,\omega)$ into the group of affine transformation: an element from $SL(X, \omega)$ is an affine transformation, but with no translation component. Since the automorphisms are a normal subgroup, if we conjugate the automorphism $f^{-1}$ by an element $g \in \Gamma$, the result is another automorphism. Call it $h$. 
 
 \begin{equation*}
    \begin{tikzcd}
        X \ar[r, dashed, "h"] \ar[d, "g"] & X \ar[d, "g"]\\
        X  \ar[r, "f^{-1}"]  & X 
    \end{tikzcd}
 \end{equation*}

 \noindent This gives us an equivalence of the sets: $$\left\{g \in \Gamma : \|g\| > R \text{ and } f^{-1}(gx) \in \mathrm{Targ}_{\phi(\lvert \lvert g \rvert \rvert)}\right\} = \left\{g \in \Gamma : \|g\| > R \text{ and } gh(x) \in \mathrm{Targ}_{\phi(\lvert \lvert g \rvert \rvert)}\right\} \text{.}$$

 Second, let $\tilde{X}$ be the full measure set such that $\left\{g \in \Gamma : g(h(x)) \in \mathrm{Targ}_{\phi(\lvert \lvert g \rvert \rvert)}\right\}$ is infinite. Observe that $h$ is not only an automorphism, but an invertible, measure-preserving transformation. Hence, $h(\tilde{X})$ is a full measure set. This second observation completes the proof. 
 \end{proof}

 Theorem \ref{thm:1} follows from Theorem \ref{thm:4}. Setting $\phi(r) = C\pi r^{-2\alpha}$ where $C$ is a constant correcting for normalization of the measure, we see that Theorem \ref{thm:1} holds for convex cocompact subgroups of the Veech group of regular square-tiled surfaces. To extend the result to all groups we employ the following lemma. 
 
 \begin{lemma}~\cite{Fin16}\label{lem:bootstrap}
 Let $\Gamma \subset SL_2(\Z)$. For any $\varepsilon > 0$, there exists a convex cocompact subgroup $\Gamma^{\prime} \subset \Gamma$ such that $\delta_{\Gamma^{\prime}} > \delta_{\Gamma} - \varepsilon$. 
 \end{lemma}
 
 \thmA*
  
 \begin{proof} For the first part of the proof, to show the set is finite for every $\alpha > \delta_{\Gamma}$, use the definition of the critical exponent, Definition \ref{critexp}. Fix $\alpha$, and for sufficiently large $N$, the asymptotics of $B_n$ (and $S_n$) are within $\varepsilon$ where $2\varepsilon < \alpha - \delta_{\Gamma}$, so the same argument as in the proof of Theorem \ref{thm:3} will work:
 
 \begin{align*}
     \sum_{g \in \Gamma} \mu(g^{-1}\mathrm{Targ}_{\phi(\lvert \lvert g \rvert \rvert)}) &\leq C \sum_{n=N+1}^{\infty} e^{(\delta_{\Gamma} + \varepsilon) n}  \phi(e^{\frac{n}{2}}) \\
     &\leq C \sum_{n=N+1}^{\infty} e^{(\alpha - \varepsilon)n}  \phi(e^{\frac{n}{2}})\text{.}\\
 \end{align*}

 \noindent We can use Cauchy condensation to deduce that the tail converges if and only if $\sum_{n=N+1}^{\infty} n^{2(\alpha - \epsilon) - 1}\phi(n)$ converges. Pick $\phi(r) = C\pi r^{-2\alpha}$, where $C$ is a constant correcting for normalization of the measure, and by Borel-Cantelli, we can conclude that for almost every $x \in X$ the set $$\left\{g \in \Gamma : gx \in \mathrm{Targ}_{\phi(\lvert \lvert g \rvert \rvert)}\right\}$$ is finite.  

 For the second part of the statement, we leverage Lemma \ref{lem:bootstrap}. For any $\varepsilon > 0$, there exists a convex cocompact subgroup $\Gamma^{\prime} \subset \Gamma$ such that $\delta_{\Gamma^{\prime}} > \delta_{\Gamma} - \varepsilon$. By applying Theorem \ref{thm:3}, we know that the set 

 \begin{equation*}
 \left\{g \in \Gamma^{\prime} \subset \Gamma : gx \in \mathrm{Targ}_{\phi(\lvert \lvert g \rvert \rvert)}\right\}
 \end{equation*}

 \noindent has infinitely many elements provided $\sum_{n=1}^{\infty} (\log n)^4 n^{-(2\delta_{\Gamma^{\prime}} + 1)} \phi(n)^{-1} < \infty$. Thus, we need 

 \begin{equation*}
 \phi(r) \geq C r^{-2\delta_{\Gamma^{\prime}}} \log^{5+\varepsilon}(r)
 \end{equation*}
 
 \noindent for all sufficiently large $r$. Recall that for our choice of $\varepsilon > 0$ above, and any $M \geq 0$, for all $r$ sufficiently large, $r^{\varepsilon} > \log^M(r)$. Since $\delta_{\Gamma} < \delta_{\Gamma^{\prime}} + \varepsilon$, for all sufficiently large $r$, we have
 
 \begin{equation*}
 r^{-2\delta_{\Gamma} + 3\varepsilon} > r^{-2(\delta_{\Gamma^{\prime}}+ \varepsilon) +3\varepsilon} = r^{-2\delta_{\Gamma^{\prime}}+ \varepsilon} > r^{-2\delta_{\Gamma^{\prime}}} \log^{5 + \varepsilon}(r) \text{.}
 \end{equation*}

 \noindent Hence, we can pick $\phi(r) = C \pi r^{-2\alpha}$ for any $C$ and any $\alpha < \delta_{\Gamma}$ and conclude that there will be infinitely many elements in the set. 
 \end{proof}

% \newpage


\begin{bibdiv}
 \begin{biblist}
 
 \bib{At09}{article}{
   title={Logarithm laws and shrinking target properties},
   author={Athreya, J.},
   journal={Proc. Indian Acad. Sci.},
   volume={119, no. 4},
   pages={541-557},
   date={2009}
   }
 
 \bib{AtMa09}{article}{
   title={Logarithm laws for unipotent flows, I},
   author={Athreya, J.},
   author={Margulis, G.},
   journal={Journal of Modern Dynamics},
   volume={3, no. 3},
   pages={359-378},
   date={2009}
   }
  
 \bib{AtMa17}{article}{
   title={Logarithm laws for unipotent flows, II},
   author={Athreya, J.},
   author={Margulis, G.},
   journal={Journal of Modern Dynamics},
   volume={11},
   pages={1-16},
   date={2017}
   }

% \bib{BagMa00}{book}{
%   title={A First Course on Representation Theory and Linear Lie Groups},
%   author={Bagchi, S.C.},
%   author={Madan, S.},
%   author={Sitaram, A.},
%   author={Tewari, U.B.},
%   date={2000},
%   publisher={Universities Press},
%   address={Hyderabad}
%   }

 \bib{Bea68}{article}{
   title={The exponent of convergence of Poincar\'e series},
   author={Beardon, R.},
   journal={Proceedings of the London Mathematical Society},
   volume={s3-18, Issue 3},
   pages={461-483},
   date={1968}
   }
 
% \bib{BeMa00}{book}{
%   title={Ergodic Theory and Topological Dynamics of Group Actions on Homogeneous Spaces},
%   author={Bekka, M.B.},
%   author={Mayer, M.},
%   date={2000},
%   publisher={University Press},
%   address={Cambridge}
%   }
 
%  \bib{BiJo97}{article}{
%   title={Hausdorff Dimension and Kleinian Groups},
%   author={Bishop, C.J.},
%   journal={},
%   volume={},
%   pages={},
%   date={}
%   } 

 \bib{BHM11}{article}{
   title={Harmonic measures versus conformal measures for hyperbolic groups},
   author={Blach\`ere, S.},
   author={Ha\"issinsky, P.},
   author={Mathieu, P.},
   journal={Annales scientifiques de l'\'Ecole Normale Sup\'erieure},
   volume={S\'erie 4, Tome 44 no. 4},
   pages={683-721},
   date={2011}
   }  

 \bib{BH99}{book}{
   title={Metric Spaces of Non-positive Curvature},
   author={Bridson, M.},
   author={Haefliger, A.},
   date={1999},
   publisher={Springer-Verlag},
   address={Berlin}
   }
   
% \bib{Bus92}{book}{
%   title={Geometry and Spectra of Compact Riemann Surfaces},
%   author={Buser, P.},
%   date={1992},
%   publisher={Birkh\"auser},
%   address={Boston}
%   }
   
% \bib{Bus94}{article}{
%   title={Some Planar Isospectral Domains},
%   author={Buser, P.},
%   author={Conway, J.},
%   author={Doyle, P.},
%   author={Semmler, K.D.},
%   journal={Math. Res. Not.},
%   number={9},
%   date={1994},
%   pages={391-400}
%   }
   
% \bib{Cha84}{book}{
%   title={Eigenvalues in Riemannian Geometry},
%   author={Chavel, I.},
%   date={1984},
%   publisher={Academic Press},
%   address={Orlando-San Diego-New York}
%   }

 \bib{Che17}{article}{
   title={Teichm\"uller Dynamics in the eyes of an algebraic geometer},
   author={Chen, D.},
   journal={Proc. Sympos. Pure Math.},
   volume={95},
   pages={171-197},
   date={2017}
   }
   
% \bib{Che79}{article}{
%   title={On the Spectral Geometry of Spaces with Cone-like singularities},
%   author={Cheeger, J.},
%   journal={Proc. Nat. Academy of Sciences},
%   date={1979}
%   }
   
% \bib{Che83}{article}{
%   title={Spectral Geometry of Singular Riemannian Spaces},
%   author={Cheeger, J.},
%   journal={J. Differential Geometry},
%   date={1983}
%   }
   
% \bib{Chu97}{book}{
%   title={Spectral Graph Theory},
%   author={Chung, F.},
%   edition={2},
%   date={1997},
%   publisher={American Mathematical Society},
%   address={Providence}
%   }

 \bib{C93}{article}{
   title={Mesures de Patterson-Sullivan sur le bord d'un espace hyperbolique au sense de Gromov},
   author={Coornaert, M.},
   journal={Pacific Journal of Mathematics},
   volume={159},
   number={2},
   pages={241-270},
   date={1993}
   }

% \bib{Co16}{book}{
%   title={Ergodic Theory and Dynamical Systems},
%   subtitle={Translation from the French version},
%   author={Coud\`ene, Y.},
%   date={2016},
%   publisher={Springer Verlag},
%   address={London}
%   }
   
% \bib{Dod81}{article}{
%   title={Eigenvalues of the Laplacian and the Heat Equation},
%   author={Dodziuk, J.},
%   journal={American Mathematical Monthly},
%   volume={88},
%   date={1981}
%   }

%  \bib{DoFaLi20}{article}{
%   title={Multiple Borel Cantelli lemma in dynamics and multilog law for recurrence},
%   author={Dolgopyat, D.},
%   author={Fayad, F.},
%   author={Liu, S.},
%   journal={preprint},
%   date={2020},
%   }

%  \bib{DoFa20}{article}{
%   title={Limit theorems for toral translations},
%   author={Dolgopyat, D.},
%   author={Fayad, F.},
%   journal={preprint},
%   date={2020},
%   }
   
% \bib{Du18}{article}{
%   title={You can hear the shape of a billiard table},
%   subtitle={Symbolic dynamics and rigidity for flat surfaces},
%   author={Duchin, M.},
%   author={Erlandsson, V.},
%   author={Leininger, C.},
%   author={Sadanand, C.},
%   date={2018}
%   }
   
% \bib{Du04}{book}{
%   title={Abstract Algebra},
%   author={Dummit, D.},
%   author={Foote, R.},
%   edition={3},
%   date={2004},
%   publisher={John Wiley \& Sons},
%   address={USA}
%   }

% \bib{DymMcK72}{book}{
%   title={Fourier Series and Integrals},
%   author={Dym, H.},
%   author={McKean, H.P.},
%   date={1972},
%   publisher={Academic Press, Inc.},
%   address={New York}
%   }
   
% \bib{Ein11}{book}{
%   title={Ergodic Theory with a view towards Number Theory},
%   author={Einsiedler, M.},
%   date={2011},
%   publisher={Springer-Verlag},
%   address={London}
%   }

% \bib{FM12}{article}{
%   title={A Primer on Mapping Class Groups},
%   author={Farb, B.},
%   author={Margalit, D.},
%   date={2012},
%   publisher={Princeton University Press},
%   address={Princeton}
%   }

  \bib{FM12}{book}{
    title={A Primer on Mapping Class Groups},
    author={Farb, B.},
    author={Margalit, D.},
    date={2012},
    publisher={Princeton University Press},
    address={41 William Street, Princeton}
    }
   
  \bib{Fa06}{article}{
   title={Mixing in the absence of the shrinking target property},
   author={Fayad, B.},
   journal={Bull. London Math. Soc.},
   volume={38 no. 5},
   pages={829-838},
   date={2006}
   }
   
 \bib{Fin16}{article}{
   title={Diophantine properties of groups of toral automorphisms},
   author={Finkelshtein, V.},
   journal={preprint},
%   volume={3},
%   pages={391-403},
   date={2016}
   }
   
% \bib{Fol99}{book}{
%   title={Real Analysis},
%   subtitle={Modern Techniques and Their Applications},
%   author={Folland, G.},
%   edition={2},
%   date={1999},
%   publisher={John Wiley \& Sons, Inc.},
%   address={New York}
%   }

% \bib{GN21}{article}{
%   title={A thank you to Bennet, Hilary, and their beloved cat, Napolean},
%   author={Goeckner (pictured), N.},
%   author={Margulis, H.},
%   author={Mozes, B.},
%   journal={Annals of Mathematics},
%   volume={147},
%   pages={93-141},
%   date={1998}
%   }

% \bib{Gor92}{article}{
%   title={One cannot hear the shape of a drum},
%   author={Gordon, C.},
%   author={Webb, D.},
%   author={Wolpert, S.},
%   journal={Bulletin of the American Mathematical Society},
%   volume={27},
%   number={1},
%   date={1992},
%   pages={134-138}
%   }

 \bib{GuJu96}{article}{
   title={The geometry and arithmetic of translation surfaces with applications to polygonal billiards},
   author={Gutkin, E.},
   author={Judge, C.},
   journal={Mathematical Research Letters},
   volume={3},
   pages={391-403},
   date={1996}
   }
   
 \bib{GuJu00}{article}{
   title={Affine Mappings of Translation Surfaces: Geometry and Arithmetic},
   author={Gutkin, E.},
   author={Judge, C.},
   journal={Duke Mathematical Journal},
   volume={103, no. 2},
   pages={191-213},
   date={2000}
   }
   
% \bib{Had45}{book}{
%   title={An Essay on the Pyschology of Invention in the Mathematical Field},
%   author={Hadamard, J.},
%   date={1945},
%   publisher={Princeton University Press},
%   address={Princeton}
%   }
   
% \bib{Hal13}{book}{
%   title={Quantum Theory for Mathematicians},
%   author={Hall, Brian C.},
%   date={2013},
%   publisher={Springer},
%   address={New York}
%   }

 \bib{HiVe95}{article}{
   title={The ergodic theory of shrinking targets},
   author={Hill, R.},
   author={Velani, S.},
   journal={Inventiones Mathematicae},
   volume={119},
   pages={175-198},
   date={1995}
   }
   
% \bib{HilJud09}{article}{
%   title={Generic Spectral Simplicity of Polygons},
%   author={Hillairet, L.},
%   author={Judge, C.},
%   journal={Proceedings of the American Mathematical Society},
%   volume={v. 137},
%   date={January 2009}
%   }

 \bib{Hil08}{article}{
   title={Spectral decomposition of square-tiled surfaces},
   author={Hillairet, L.},
   journal={Mathematische Zeitschrift},
   volume={vol. 260, no. 2},
   date={2008}
   }

 \bib{Hil09}{article}{
   title={Spectral theory of translation surfaces: A short introduction},
   author={Hillairet, L.},
   journal={S\'eminaire de th\'eorie spectrale et g\'eom\'etrie},
   volume={vol. 28},
   date={2009-2010},
   pages={51-62}
   }

 \bib{HubSch04}{article}{
   title={An introduction to Veech surfaces},
   author={Hubert, P.},
   author={Schmidt, T.},
   editor={Katok, A.},
   editor={Hasselblatt, B.},
   journal={Handbook of Dynamical Systems},
   volume={vol. 1B},
   publisher={Elsevier},
   date={2006}
   }

% \bib{Hum72}{book}{
%   title={Introduction to Lie Algebras and Representation Theory},
%   author={Humphreys, J.},
%   date={1972},
%   publisher={Springer-Verlag},
%   address={New York}
%   }
   
% \bib{Ike83}{article}{
%   title={On spherical space forms which are isospectral but not isometric},
%   author={Ikeda, A.},
%   journal={J. Math. Soc.},
%   address={Japan},
%   volume={35},
%   date={1983},
%   pages={437-444}
%   }
   
% \bib{Kac66}{article}{
%   title={Can one hear the shape of a drum?},
%   author={Kac, M.},
%   journal={Amer. Math. Monthly},
%   volume={73},
%   date={1966},
%   pages={1-23}
%   }
   
 \bib{K92}{book}{
   title={Fuchsian Groups},
   author={Katok, S.},
   date={1992},
   publisher={The University of Chicago Press},
   address={Chicago}
   }

% \bib{Ke59}{article}{
%   title={Symmetric random walks on groups},
%   author={Kesten, H.},
%   journal={Transactions of the American Mathematical Society},
%   volume={92, no. 2},
%   date={1959},
%   pages={336-354}
%   }
 
 \bib{KlMa99}{article}{
   title={Logarithm laws for flows on homogeneous spaces},
   author={Kleinbock, D.},
   author={Margulis, G.},
   journal={Inventiones Mathematicae},
   volume={138, no. 3},
   date={1999},
   pages={451-494}
   }

% \bib{Kna86}{book}{
%   title={Representation Theory of Semisimple Groups},
%   author={Knapp, A.},
%   date={1986},
%   publisher={Princeton University Press},
%   address={Princeton}
%   }
   
% \bib{LeeRM19}{book}{
%   title={Introduction to Riemannian Geometry},
%   author={Lee, J.},
%   edition={2},
%   date={2019},
%   publisher={Springer},
%   address={New York}
%   }
   
% \bib{LeeSM13}{book}{
%   title={Introduction to Smooth Manifolds},
%   author={Lee, J.},
%   edition={2},
%   date={2013},
%   publisher={Springer},
%   address={New York}
%   }
   
% \bib{LeeTM11}{book}{
%   title={Introduction to Topological Manifolds},
%   author={Lee, J.},
%   edition={2},
%   date={2011},
%   publisher={Springer},
%   address={New York}
%   }

 \bib{Ma18}{article}{
   title={Three lectures on square-tiled surfaces},
   author={Matheus, C.},
%   journal={Pacific Journal of Mathematics},
%   volume={vol. 20, no. 1},
   date={2018}
   }
   
% \bib{Mc99}{article}{
%   title={Hausdorff dimension and conformal dynamics I: strong convergence of Kleinian groups},
%   author={McMullen, C.},
%   journal={J. Differential Geom.},
%   volume={51, no. 3},
%   date={1999},
%   pages={471-515}
%   }
   
% \bib{McM03}{article}{
%   title={Dynamics of $SL_2(\R)$ over Moduli Space in Genus Two},
%   author={McMullen, C.},
%   journal={Ann. of Math.},
%   volume={165},
%   date={2007},
%   pages={397-456}
%   }
  
% \bib{Mil64}{article}{
%   title={Eigenvalues of the Laplace operator on certain manifolds},
%   author={Milnor, J.},
%   journal={Proc. Nat. Acad. Sci.},
%   volume={51},
%   date={1964},
%   pages={542}
%   }

 \bib{Pat76}{article}{
   title={The exponent of convergence for Poincar\'e series},
   author={Patterson, S.J.},
   journal={Monatshefte f\"ur Mathematik},
   volume={82},
   pages={297-315},
   date={1976}
   }
 
 \bib{Phi67}{article}{
   title={Some metrical theorems in number theory},
   author={Philipp, W.},
   journal={Pacific Journal of Mathematics},
   volume={vol. 20, no. 1},
   date={1967}
   }

 \bib{Q}{article}{
   title={An overview of Patterson-Sullivan theory},
   author={Quint, J.F.},
%   journal={Pacific Journal of Mathematics},
%   volume={vol. 20, no. 1},
%   date={2018}
   }
   

% \bib{Ros97}{book}{
%   title={The Laplacian on a Riemannian Manifold},
%   author={Rosenberg, S.},
%   date={1992},
%   publisher={Cambridge University Press},
%   address={Cambridge}
%   }
   
% \bib{Rud73}{book}{
%   title={Functional Analysis},
%   author={Rudin, W.},
%   date={1973},
%   publisher={McGraw-Hill},
%   address={United States of America}
%   }
   
% \bib{Sch14}{book}{
%   title={A Course in Complex Analysis and Riemann Surfaces},
%   author={Schlag, W.},
%   date={2014},
%   publisher={American Mathematical Society},
%   address={Providence}
%   }
   
% \bib{Ste15}{book}{
%   title={Calculus},
%   author={Stewart, J.},
%   edition={8},
%   date={2015},
%   publisher={Cengage Learning},
%   address={Boston}
%   }

 \bib{Sul79}{article}{
   title={The density at infinity of a discrete group of hyperbolic motions},
   author={Sullivan, D.},
   journal={Publications math\'ematiques de l’I.H.\'E.S.},
   volume={50},
   date={1979},
   pages={171-202}
   }

 \bib{Sul82}{article}{
   title={Disjoint spheres, approximation by imaginary quadratic numbers, and the logarithm law for geodesics},
   author={Sullivan, D.},
   journal={Acta Math.},
   volume={149},
   date={1982},
   pages={215-237}
   }
   
% \bib{Sun85}{article}{
%   title={Riemannian coverings and isospectral manfifolds},
%   author={Sunada, T.},
%   journal={Ann. of Math.},
%   volume={(2) 121},
%   date={1985},
%   pages={169-186}
%   }

% \bib{Str08}{book}{
%   title={Partial Differential Equations},
%   subtitle={An Introduction},
%   author={Strauss, W.},
%   edition={2},
%   date={2008},
%   publisher={John Wiley \& Sons, Inc.},
%   address={Hoboken}
%   }

% \bib{Ter13}{book}{
%   title={Harmonic Analysis on Symmetric Spaces and Applications I},
%   author={Terras, A.},
%   edition={2},
%   date={2013},
%   publisher={Springer},
%   address={New York}
%   }

 \bib{Th88}{article}{
   title={On the geometry and dynamics of diffeomorphisms of surfaces},
   author={Thurston, W.},
   journal={Bulletin of the Amer. Math. Soc.},
   volume={19 no. 2},
   pages={417-431},
   date={1988}
   }

% \bib{Tse08}{article}{
%   title={More remarks on shrinking target properties},
%   author={Tseng, J.},
%   journal={Ann. of Math.},
%   volume={(2) 121},
%   date={2008},
%   pages={169-186}
%   }
   
% \bib{Vig80}{article}{
%   title={Vari\'et\'es riemanniennes isospectrales et non isom\'etriques},
%   author={Vign\'eras, M. F.},
%   journal={Ann. of Math.},
%   volume={112},
%   date={1980},
%   pages={21-32}
%   }

% \bib{Vog87}{book}{
%   title={Unitary Representations of Reductive Lie Groups},
%   author={Vogan, D.},
%   date={1987},
%   publisher={Princeton University Press},
%   address={Princeton}
%   }
 
% \bib{Wa82}{book}{
%   title={Introduction to Ergodic Theory},
%   author={Walters, P.},
%   date={1982},
%   publisher={Spring-Verlag},
%   address={New York}
%   }
 
   \bib{Wri15}{article}{
   title={Translation surfaces and their orbit closures},
   subtitle={An introduction for a broad audience},
   author={Wright, A.},
   journal={EMS Surv. Math. Sci.},
   date={2015}
   }
   
   \bib{Vee89}{article}{
   title={Teichm\"uller curves in moduli space, Eisenstein series and an application to triangular billiards},
   author={Veech, W.A.},
   journal={Inventiones Mathematicae},
   volume={97},
   pages={553-583},
   date={1989}
   }
   
%   \bib{Vor96}{article}{
%   title={Plane Structures and billiards in rational polygons: the Veech alternative},
%   author={Vorobets, Ya.},
%   journal={Russ. Math. Surv.},
%   volume={51},
%   pages={779-817},
%   date={1996}
%   }
   
% \bib{Zim90}{book}{
%   title={Essential Results of Functional Analysis},
%   author={Zimmer, R.},
%   date={1990},
%   publisher={The University of Chicago Press},
%   address={Chicago}
%   }
 
 \bib{Zm11}{article}{
   title={Origamis and permutation groups},
   author={Zmiaikou, D.},
   journal={PhD Thesis, University Paris-Sud},
%   volume={Vol. 1},
   date={2011}
   }
 
 \bib{Zor06}{article}{
   title={Flat surfaces},
   author={Zorich, A.},
   journal={Frontiers in Number Theory, Physics, and Geometry},
   volume={vol. 1},
   date={2006}
   }

 \end{biblist}
 \end{bibdiv}
 \end{document}